\documentclass[12pt,psfig,reqno]{amsart}
\usepackage{mathrsfs}
\usepackage{txfonts}
\usepackage{cite}
\usepackage{amsmath}
\usepackage{bbm}
\usepackage{}
\usepackage{amssymb}
\usepackage{amsfonts}
\usepackage{amscd}
\usepackage{amssymb,amsfonts,latexsym,enumerate}
\usepackage{latexsym,amsmath,amssymb,amsthm,amsfonts,arydshln,enumerate,graphicx}
\usepackage{graphics,verbatim}
\usepackage{graphicx}

\usepackage[usenames]{color} 
\makeatletter

\renewcommand\section{\@startsection {section}{1}{\z@}%
                                   {-3.5ex \@plus -1ex \@minus -.2ex}%
                                   {2.3ex \@plus.2ex}%
                                   {\centering\normalfont\bf}}

 \numberwithin{equation}{section}
\numberwithin{equation}{section}

\numberwithin{equation}{section}

\theoremstyle{plain}
\newtheorem{thm}{Theorem}[section]
\newtheorem{lemma}[thm]{Lemma}
\newtheorem{pro}[thm]{Proposition}

\newtheorem{ex}[thm]{Example}
\newtheorem{de}[thm]{Definition}
\newtheorem{re}[thm]{Remark}

\newtheorem*{co*}{Conjecture}
\newtheorem*{thm*}{Theorem}

\begin{document}
\title[Fourier bases of self-affine measures]{Fourier bases of a class of planar self-affine measures}
\author{Ming-Liang Chen,  Jing-Cheng Liu$^*$ and Zhi-Yong Wang}
\address{School of Mathematics and Computer Science, Gannan Normal University, Ganzhou 341000, P.R.  China}
\email{mathcml@163.com}
\address{Key Laboratory of Computing and Stochastic Mathematics (Ministry of Education), College of Mathematics and Statistics, Hunan Normal University, Changsha, Hunan 410081,  P.R. China}
\email{jcliu@hunnu.edu.cn}
\address{College of Mathematics and Computational Science, Hunan First Normal University, Changsha, Hunan 410205, P.R. China}
\email{wzyzzql@163.com}

\date{\today}
\thanks{ The research is supported in part by the NNSF of China (Nos. 12071125, 12001183, 11831007, and 11971500) and the Hunan Provincial NSF (No. 2020JJ5097).\\
$^*$Corresponding author.}

\begin{abstract}
Let $\mu_{M,D}$ be the  planar self-affine measure generated by an expansive integer matrix $M\in M_2(\mathbb{Z})$ and a non-collinear integer digit set
$D=\left\{\begin{pmatrix}
0\\0\end{pmatrix},\begin{pmatrix}
\alpha_{1}\\ \alpha_{2}
\end{pmatrix},
\begin{pmatrix}
\beta_{1}\\ \beta_{2}
\end{pmatrix},
\begin{pmatrix}
-\alpha_{1}-\beta_{1}\\ -\alpha_{2}-\beta_{2}
\end{pmatrix}\right\}$.  In this paper, we show that $\mu_{M,D}$ is a spectral measure if and only if there exists a matrix  $Q\in M_2(\mathbb{R})$  such that  $(\tilde{M},\tilde{D})$ is admissible, where $\tilde{M}=QMQ^{-1}$ and $\tilde{D}=QD$. In particular, when $\alpha_1\beta_2-\alpha_2\beta_1\notin 2\Bbb Z$, $\mu_{M,D}$ is a spectral measure if and only if $M\in M_2(2\mathbb{Z})$.
\end{abstract}

\maketitle

\section{\bf Introduction\label{sect.1}}
Let $\mu$ be a Borel probability measure with compact support on $\mathbb{R}^n$, and let $\langle\cdot,\cdot\rangle$ denote the standard inner product on $\mathbb{R}^n$. We say that $\mu$ is a \emph{spectral measure} if there exists a countable set $\Lambda\subset\mathbb{R}^n$ such that  the exponential function system $E_\Lambda:=\{e^{2\pi i\langle\lambda,x\rangle}:\lambda\in\Lambda\}$ forms an orthonormal basis for  the Hilbert space $L^2(\mu)$. In this case, we call $\Lambda$  a \emph{spectrum} of $\mu$ and  $(\mu,\Lambda)$  a \emph{spectral pair}.
In particular, if $\mu$ is the normalized Lebesgue measure supported on a Borel set $\Omega$, then $\Omega$ is called a \emph{spectral set}.

Spectral measure is a natural generalization of spectral set introduced by Fuglede \cite{Fuglede_1974}, who proposed the famous conjecture: $\Omega$ is a spectral set if and only if $\Omega$ is a translational tile. It is known \cite{HLL} that a spectral measure $\mu$ must be of pure type: $\mu$ is either discrete, or absolutely continuous or singularly continuous. The first singularly continuous spectral measure was constructed by Jorgensen and Pedersen in 1998 \cite{Jorgensen-Pedersen_1998}. They  proved that the middle-fourth Cantor measure is a spectral measure with  spectrum
$$\Lambda=\left\{\sum_{k=0}^n 4^{k}\ell_k: \ell_k\in \{0,1\},n\in\mathbb{N}\right\}.$$
Following this discovery, there is a considerable number of papers on the spectrality of self-similar/self-affine measures and the construction of their spectra,  see \cite{An-He-Tao_2015,Dai_2012,Dai-He-Lau_2014,Deng-Lau_2015,Dutkay-Jorgensen_2007,
Dutkay-Jorgensen_2007_1,Dutkay-Haussermann-Lai_2019,FHL,Laba-Wang_2002}  and the references  therein. These results are generalized further to  some classes of Moran measures (see e.g.,
\cite{An-He_2014,DL,Fu-He-Wen}), and some surprising convergence properties of the associated Fourier series were discovered in \cite{Strichartz_2000,S_06}. These fractal measures also have very close connections with the theory of multiresolution analysis in wavelet analysis, see
 \cite{DJ_06}.

In \cite{DuJ09}, Dutkay and Jorgensen summarized some known results regarding iterated function systems (IFS, see \cite{Hutchinson_1981}  for details). Two approaches to harmonic analysis on IFS  have been popular: one based on a discrete version of the more familiar and classical second order Laplace differential operator of potential theory, see \cite{K04,KSW01,LNRG}; and the other is based on Fourier series. The first model in turn is motivated by infinite discrete network of resistors, and the harmonic functions are defined by minimizing a global measure of resistance, but this approach does not rely on Fourier series. In contrast, the second approach begins with Fourier series, and it has its classical origins in lacunary Fourier series \cite{Ka86}.

Let $\{\phi_d(x)\}_{d\in D}$ be an \emph{iterated function system} (IFS) defined by $\phi_d(x)=M^{-1}(x+d)~(x\in \mathbb{R}^n,d\in D)$, where
$M\in M_n(\mathbb{R})$ is an expansive real matrix (that is, all the eigenvalues of $M$ have modulus strictly greater than one) and $D\subset\mathbb{R}^n$ is a finite digit set with cardinality $\#D$.  It is known \cite{Hutchinson_1981} that there exists a unique probability measure $\mu_{M,D}$   satisfying that
\begin{equation}\label{1.1}
\mu_{M,D}=\frac{1}{\#D}\sum_{d\in D}\mu_{M,D}\circ\phi_d^{-1}.
\end{equation}
This measure is supported on  $T(M,D)$, which is the unique nonempty compact set satisfying that $T(M,D)=\bigcup_{d\in D}\phi_d(T(M,D))$. Alternatively, it can also be described as a set of radix expansion with base $M$, that is,
$$T(M,D)=\left\{\sum_{k=1}^\infty M^{-k}d_k: d_k\in D\right\}:=\sum_{k=1}^\infty M^{-k}D.$$
We call $\mu_{M,D}$ a \emph{self-affine measure} and $T(M,D)$ a \emph{self-affine set}, respectively. In particular, if $M$ is a multiple of an orthonormal matrix, then $\mu_{M,D}$ and $T(M,D)$ are called {\it self-similar measure} and  {\it self-similar set}, respectively.
It is known that a self-affine measure $\mu_{M,D}$ can be expressed by the infinite convolution  of discrete measures as follows:
\begin{equation*}
\mu_{M,D}=\delta_{M^{-1}D}*\delta_{M^{-2}D}*\delta_{M^{-3}D}*\cdots,
\end{equation*}
where $*$ is the convolution sign, $\delta_E=\frac{1}{\#E}\sum_{e\in E}{\delta_e}$ for a finite set $E$ and $\delta_e$ is the Dirac measure at the point $e$.

Self-affine measures have the advantage that their Fourier transforms (see \eqref{2.1}) can be explicitly written down as an infinite product, which allows us to compute their zeros. The previous research on self-affine measures $\mu_{M,D}$ and its Fourier transform  have revealed some surprising connections with a number of areas in mathematics such as harmonic analysis, dynamical system, number theory and others (see, e.g., \cite{GM,JKS,S_1994}).
For a self-affine measure  $\mu_{M,D}$,  it is difficult to verify that $\mu_{M,D}$ is not a spectral measure. Some results are available in this direction: it was proved \cite{Jorgensen-Pedersen_1998} that the middle-third Cantor set does not have more than two mutually orthogonal exponentials.  In \cite{LaW06}, {\L}aba and Wang  proved that a class of absolutely continuous measures with good decay of the Fourier transform are not spectral, and that if the diameter of the support set of an absolutely continuous measure is not much larger than the measure, then the set tiles (and is spectral) by a lattice. In \cite{Dutkay-Jorgensen_2007}, a large number of self-affine  Sierpinski-type measures are investigated and shown to be spectral or non-spectral. In \cite{Li07}, conditions are given involving the prime factorization of $\det(M)$ to guarantee that the corresponding measure   $\mu_{M,D}$ does not have certain spectra.

The construction of self-affine spectral measures usually stems from the existence of Hadamard triple  (known also as compatible pair). The appearance of Hadamard triple stems from the terminology of \cite{Strichartz_2000}.

\begin{de}\label{deA}
{\rm Let $M\in M_n(\mathbb{Z})$ be an expansive integer matrix, and let $D,S\subset\mathbb{Z}^n$ be two finite digit
sets with  $\#D=\#S=N$. We say
that $(M, D)$ is {\it admissible} (or $(M^{-1}D, S)$ forms a {\it compatible
pair} or $(M,D,S)$ forms a {\it Hadamard triple}) if the matrix
$$
H=\frac{1}{\sqrt{N}}\begin{pmatrix}e^{2\pi i\langle M^{-1}d, s\rangle}\end{pmatrix}_{d\in D, s\in S}
$$
is unitary, i.e., $H^*H=I$, where $I$ is a $n\times n$ identity matrix.}
\end{de}

The well-known result of Jorgensen and Pedersen \cite{Jorgensen-Pedersen_1998} shows that
if $(M, D)$ is admissible, then there are infinite families of orthogonal exponential functions in $L^2(\mu_{M,D})$. Moreover,
Dutkay and Jorgensen \cite{Dutkay-Jorgensen_2007_1,Dutkay-Jorgensen_2009} formulated the  famous conjecture:  if $(M, D)$ is admissible, then $\mu_{M,D}$ is a spectral measure.
It was first proved in one dimension by {\L}aba and Wang [25].  Moreover, the conjecture is true in high  dimensions under some additional assumptions, introduced by Strichartz [35].  There are many other papers investigated it in  high dimensional cases, see \cite{Dutkay-Jorgensen_2007,Li_2013}. In the end, Dutkay, Haussermann and Lai  \cite{Dutkay-Haussermann-Lai_2019} proved that

\begin{thm}\label{th(DHL)}
Let $M\in M_n(\mathbb{Z})$ be an  expansive integer matrix, and let $D\subset\mathbb{Z}^n$ be a finite digit
set. If $(M, D)$ is admissible, then  $\mu_{M,D}$ is a spectral measure.
\end{thm}

In \cite{FHL}, Fu, He and Lau  gave an example to illustrate that the sufficient condition in Theorem \ref{th(DHL)}
is not necessary in one dimension. For an expansive integer matrix $M\in M_2(\mathbb{Z})$ and the classic digit set $D=\left\{\begin{pmatrix}
0\\0\end{pmatrix},\begin{pmatrix}
1\\ 0
\end{pmatrix},
\begin{pmatrix}
0\\ 1
\end{pmatrix}\right\}$,
the spectrality and non-spectrality of the corresponding self-affine measure $\mu_{M,D}$ have been widely investigated by many researchers, see \cite{Dutkay-Jorgensen_2007,Li_2008,Li_2013}. Eventually,  An, He and Tao \cite{An-He-Tao_2015} completely settled the spectrality of $\mu_{M,D}$.  More precisely, they showed that $\mu_{M,D}$ is a spectral measure if and only if $(M, D)$ is admissible. For a more general integer digit set $D$ with $0\in D$ and $\#D=3$, there is also a complete spectral characterization, see \cite{CL,LW,LZWC}.
In addition to these, another important integer digit set is
\begin{equation}\label{1.2}
D =\left\{\begin{pmatrix}
0\\0\end{pmatrix},\begin{pmatrix}
\alpha_{1}\\ \alpha_{2}
\end{pmatrix},
\begin{pmatrix}
\beta_{1}\\ \beta_{2}
\end{pmatrix},
\begin{pmatrix}
-\alpha_{1}-\beta_{1}\\ -\alpha_{2}-\beta_{2}
\end{pmatrix}\right\},
\end{equation}
where $\alpha_1\beta_2-\alpha_2\beta_1\neq0$. The existence of infinitely many orthogonal exponentials in $L^2(\mu_{M,D})$ has been fully studied by  \cite{Li-2014,Su-Liu-Liu_2019,SWC}.
However, to the best of our knowledge, the complete description of spectral properties is not known yet, even for the special case where $\alpha_{1}=1$, $\alpha_{2}=0$, $\beta_{1}=0$ and $\beta_{2}=1$. A natural subsequent question is that

{\bf (Qu 1):} For an expansive integer matrix $M\in M_2(\mathbb{Z})$  and the digit set $D$ given by \eqref{1.2}, what is the sufficient and necessary condition for $\mu_{M,D}$ to be a spectral measure?

In this paper, our main purpose is to initiate a study on the question {\bf (Qu 1)}. It is difficult to answer this question by the known methods. Until recently, we get inspiration from  \cite{DL}, which considers the spectrality of a class of Moran-type self-similar measures.
Before presenting our results, a reasonable assumption for the digit set $D$ is necessary. Without loss of generality, we can assume that $\gcd(\alpha_1,\alpha_2,\beta_1,\beta_2)=1$  by Lemma \ref{th(2.2)}.

Our first main result is as follows:

\begin{thm}\label{th(main)}
Let  $\mu_{M,D}$  be defined by \eqref{1.1}, where $M\in M_2(\Bbb Z)$ is an expansive integer matrix and $D$ is given by \eqref{1.2}. Then  $\mu_{M,D}$ is a spectral measure if and only if there exists a matrix  $Q\in M_2(\mathbb{R})$  such that  $(\tilde{M},\tilde{D})$ is admissible, where $\tilde{M}=QMQ^{-1}$ and $\tilde{D}=QD$.
\end{thm}

We remark that Theorem \ref{th(main)} gives a complete answer to the spectral question {\bf (Qu 1)}. We now outline the strategy of the proof of Theorem \ref{th(main)}.
The sufficiency of Theorem  \ref{th(main)} follows directly from Theorem \ref{th(DHL)} and Lemma \ref{th(2.2)}.
The more challenging part of the proof is the necessity.
The key point  is to construct a self-affine measure $\mu_{\tilde{M},\tilde{D}}$  so that it has the same spectrality as the measure $\mu_{M,D}$, and then the necessity follows immediately from  Theorems  \ref{th(1.5)} and \ref{th(1.6)}. What is exciting is that the proof method of  the necessity is new and completely different from the previous work proving spectral self-affine measures.

It is worth noting that if $D$ satisfies $\alpha_1\beta_2-\alpha_2\beta_1\notin 2\Bbb Z$, we can give more explicit sufficient and necessary conditions for $\mu_{M,D}$ to be a spectral measure. Before presenting them, some notations are needed.  For any  integer $p\geq 2$, we denote
\begin{equation}\label{1.3}
\mathcal{F}_p^2:=\frac{1}{p}\left\{\begin{pmatrix}
l_1 \\
l_2
\end{pmatrix}:0\leq l_1,l_2\leq p-1, l_i\in\mathbb{Z}\right\} \quad {\rm and} \quad \mathring{\mathcal{F}}_p^2:=\mathcal{F}_p^2\setminus\{\bf{0}\}.
\end{equation}
Under the above notations and the assumption of $\alpha_1\beta_2-\alpha_2\beta_1\notin 2\Bbb Z$, we give the second main result:

\begin{thm}\label{th(main2)}
Let  $\mu_{M,D}$ and $\mathring{\mathcal{F}}_p^2$  be defined by \eqref{1.1} and \eqref{1.3} respectively, where $M\in M_2(\Bbb Z)$ is an expansive integer  matrix and $D$ is given by \eqref{1.2}. If $\alpha_1\beta_2-\alpha_2\beta_1\notin 2\Bbb Z$, then the following statements are equivalent:
 \begin{enumerate}[(i)]
 	\item   $\mu_{M,D}$ is a spectral measure;
 \item $M\in M_2(2\Bbb Z)$;
 \item $M\mathring{\mathcal{F}}_2^2\subset\mathbb{Z}^2$;
  \item $(M,D)$ is admissible.
 \end{enumerate}
\end{thm}

We point out that the proofs of Theorems \ref{th(main)} and \ref{th(main2)}  are based on the precise  form of the matrix $\tilde{M}$ in Theorem \ref{th(main)}.
Before giving the form, some technical work needs to be done. For an expansive integer matrix $M\in M_2(\mathbb{Z})$ and the digit set $D$ given by \eqref{1.2},  we can let $M=\begin{pmatrix}
a&b \\
c&d
\end{pmatrix}$ and $\alpha_1\beta_2-\alpha_2\beta_1=2^\eta\gamma$ with $\eta\geq0$ and $\gamma\notin2\mathbb{Z}$. Without loss of generality, we assume $\gcd(\alpha_1,\alpha_2)=\alpha$ with $\alpha\notin2\mathbb{Z}$ (otherwise, we can  choose $\alpha=\gcd(\beta_1,\beta_2)$ with $\alpha\notin2\mathbb{Z}$ since $\gcd(\alpha_1,\alpha_2,\beta_1,\beta_2)=1$).   Let $\alpha_1=\alpha t_1$ and $\alpha_2=\alpha t_2$ with $\gcd(t_1,t_2)=1$, then there exist $p,q\in\mathbb{Z}$ such that $pt_1+qt_2=1$. Clearly, $\alpha=p\alpha_1+q\alpha_2$ and $\alpha|\gamma$. For convenience, we denote $\omega=p\beta_1+q\beta_2$ and $\beta=\gamma/\alpha$. It is easy to check that $t_1\alpha_2=t_2\alpha_1$ and $t_1\beta_2-t_2\beta_1=2^\eta\beta$ with $\beta\notin2\mathbb{Z}$.
Denote $Q=\begin{pmatrix}
 p&q \\
 -t_2&t_1
\end{pmatrix}$. Then one has
\begin{equation}\label{1.4}
\tilde{M}:=QMQ^{-1}
=\begin{pmatrix}
(pa+qc)t_1+(pb+qd)t_2 & (pb+qd)p-(pa+qc)q \\
(ct_1-at_2)t_1+(dt_1-bt_2)t_2 & (dt_1-bt_2)p-(ct_1-at_2)q
\end{pmatrix}
\end{equation}
and
\begin{equation}\label{1.5}
\begin{aligned}
\tilde{D}:&=QD=\left\{ \begin{pmatrix}
0\\0\end{pmatrix},\begin{pmatrix}
\alpha\\0\end{pmatrix},\begin{pmatrix}
\omega \\ 2^\eta\beta\end{pmatrix},\begin{pmatrix}
-\alpha-\omega \\
-2^\eta\beta\end{pmatrix}\right\}\subset\mathbb{Z}^2.
\end{aligned}
\end{equation}
Obviously, $\tilde{M}$ is an expansive integer matrix with $\det(\tilde{M})=\det(M)$. Moreover, $\eta=0$ and $\eta>0$ are equivalent to $\alpha_1\beta_2-\alpha_2\beta_1\notin 2\Bbb Z$ and $\alpha_1\beta_2-\alpha_2\beta_1\in 2\Bbb Z$, respectively.

For $\eta=0$ in $\tilde{D}$, we have the following conclusion, which  is equivalent to  Theorem \ref{th(main2)} by using the property of similarity transformation.

\begin{thm}\label{th(1.5)}
Let  $\mu_{\tilde{M},\tilde{D}}$  and $\mathring{\mathcal{F}}_p^2$  be defined by \eqref{1.1} and \eqref{1.3} respectively, where $\tilde{M}$ and $\tilde{D}$ are given by \eqref{1.4} and \eqref{1.5} respectively. If $\eta=0$, then the following statements are equivalent:
 \begin{enumerate}[(i)]
 \item   $\mu_{\tilde{M},\tilde{D}}$ is a spectral measure;
 \item $\tilde{M}\in M_2(2\Bbb Z)$;
 \item $\tilde{M}\mathring{\mathcal{F}}_2^2\subset\mathbb{Z}^2$;
 \item $(\tilde{M},\tilde{D})$ is admissible.
 \end{enumerate}
\end{thm}

On the other hand, if $\eta>0$  in $\tilde{D}$, the form of $\tilde{M}$ is different from that in the case $\eta=0$.

\begin{thm}\label{th(1.6)}
Let  $\mu_{\tilde{M},\tilde{D}}$  be defined by \eqref{1.1}, where $\tilde{M}$ and $\tilde{D}$ are given by \eqref{1.4} and \eqref{1.5} respectively. If $\eta>0$, then  $\mu_{\tilde{M},\tilde{D}}$ is a spectral measure if and only if $\tilde{M}=\begin{pmatrix}
\tilde{ a}&\tilde{b} \\
\tilde{c}&\tilde{d}
\end{pmatrix}$ satisfies $\tilde{a},\tilde{d}\in2\Bbb Z$ and $2^{\eta+1}| \tilde{c}$.
\end{thm}

We now give a brief explanation of the proofs of Theorems \ref{th(1.5)} and  \ref{th(1.6)}. The main  technical difficulty in the proofs   lies in ``$(i)\Rightarrow(ii)$" of Theorem \ref{th(1.5)} and the necessity of Theorem \ref{th(1.6)}. More precisely, the key point  is to construct a Moran measure $\mu_{A,\tilde{M},\tilde{D}}$ (see \eqref{3.1}) so that it has the same spectral property as $\mu_{\tilde{M},\tilde{D}}$. For the matrix $A$, we need to cleverly  describe its complete residue system (Proposition \ref{th(3.3)}). Furthermore, we  carefully investigate the structure of the spectrum of $\mu_{A,\tilde{M},\tilde{D}}$ (see \eqref{3.11}).  And then we get a property of decomposition on the spectrum of $\mu_{\tilde{M},\tilde{D}}$  under the assumption  that $\mu_{A,\tilde{M},\tilde{D}}$  is a spectral measure (Lemma \ref{th(3.5)}). With the help of these, the proof becomes within reach.

The paper is organized as follows. In Section 2, we introduce some basic definitions and lemmas. In Section 3, we focus on proving Theorems \ref{th(1.5)} and \ref{th(1.6)}. Finally, we prove Theorems  \ref{th(main)} and \ref{th(main2)}, and give some concluding remarks in Section 4.

\section{\bf Preliminaries\label{sect.2}}

The aim in this section is to collect some necessary definitions and basic facts for self-affine measures. For the self-affine measure $\mu_{M,D}$ defined by \eqref{1.1}, the Fourier transform of $\mu_{M,D}$ is defined as usual,
\begin{equation}\label{2.1}
\hat{\mu}_{M,D}(\xi)=\int e^{2\pi i\langle x,\xi\rangle }d\mu_{M,D}(x)=\prod_{j=1}^\infty m_D({M^{*}}^{-{j}}\xi), \quad  \xi\in\mathbb{R}^n,
\end{equation}
where $M^*$ denotes the transpose of $M$ and $m_{D}(\cdot)=\frac{1}{\#D}\sum_{d\in D}{e^{2\pi i\langle d,\cdot\rangle}}$ is the mask polynomial of $D$.
It is easy to see that $m_D(\cdot)$ is a $\mathbb{Z}^n$-periodic function if $D\subset \mathbb{Z}^n$.

Let $\mathcal{Z}(f)=\{x:f(x)=0\}$ be the zero set  of a function $f$.
It follows from \eqref{2.1} that
\begin{equation}\label{2.2}
\mathcal{Z}(\hat{\mu}_{M, D})=\bigcup_{j=1}^{\infty}M^{*j}(\mathcal{Z}(m_D)).
\end{equation}
For any distinct $\lambda_1,\lambda_2\in \mathbb{R}^n$, the orthogonality condition means that
$$
0=\langle e^{2\pi i \langle \lambda_1,x\rangle},e^{2\pi i \langle \lambda_2,x\rangle}\rangle_{L^2(\mu_{M,D})}=\int e^{2\pi i \langle\lambda_1-\lambda_2,x\rangle}d\mu_{M,D}(x)=\hat{\mu}_{M,D}(\lambda_1-\lambda_2).
$$
It is easy to see that for a countable subset $\Lambda\subset\mathbb{R}^n$,  $E_\Lambda=\{e^{2\pi i\langle\lambda,x\rangle}:\lambda\in\Lambda\}$ is an orthogonal family of $L^2(\mu_{M,D})$ if and only if
 \begin{equation}\label{2.3}
 (\Lambda-\Lambda)\setminus\{0\}\subset\mathcal{Z}(\hat{\mu}_{M,D}).
\end{equation}
We call $\Lambda$  an {\it orthogonal set} (respectively, {\it spectrum}) of $\mu_{M,D}$ if $E_\Lambda$ forms an orthogonal family (respectively, Fourier basis) of $L^2(\mu_{M,D})$. As the properties of spectra are invariant under  a translation, one may assume that $0\in\Lambda$, and hence $\Lambda\subset\Lambda-\Lambda$.

In a number of applications, one encounters a measure $\mu$ and a subset $\Lambda$  such that the functions $e^{2\pi i\langle\lambda,x\rangle}$ indexed by $\Lambda$ are orthogonal in $L^2(\mu)$, but a separate argument is needed in
order to show that the family is complete. Let
\begin{equation}\label{2.4}
Q_{\mu,\Lambda}(\xi)=\sum_{\lambda\in\Lambda}|{\hat{\mu}(\xi+\lambda)}|^2,\quad \xi\in\mathbb{R}^n.
\end{equation}
The well known result of Jorgensen and Pedersen \cite[Lemma 4.2]{Jorgensen-Pedersen_1998} shows that $Q_{\mu,\Lambda}(\xi)$ is an entire function if $\Lambda$  is an orthogonal set of $\mu$. The following provides a universal test which allows us to decide whether an orthogonal set $\Lambda$ is a spectrum of  the measure $\mu$.

\begin{thm}\cite{Jorgensen-Pedersen_1998}\label{th(JP)}
Let $\mu$ be a Borel probability measure with compact support on $\mathbb{R}^n$, and let $\Lambda\subset\mathbb{R}^n$ be a countable set. Then
 \begin{enumerate}[(i)]
 	\item   $\Lambda$ is an orthogonal set of $\mu$ if and only if $Q_{\mu,\Lambda}(\xi)\leq1$ for $\xi\in\mathbb{R}^n$;
 	\item
$\Lambda$ is a spectrum of $\mu$  if and only if $Q_{\mu,\Lambda}(\xi)\equiv1$ for $\xi\in\mathbb{R}^n$.
 \end{enumerate}
\end{thm}

The following lemma indicates that the spectral property of
$\mu_{M,D}$ is invariant under a similarity transformation.

\begin{lemma}\cite{Dutkay-Jorgensen_2007}\label{th(2.2)}
Let $D_1, D_2\subset \Bbb R^n$ be two finite digit sets with the same cardinality,  and let $M_1, M_2\in M_n(\Bbb R)$ be two expansive real matrices. If there exists a matrix  $Q\in M_n(\mathbb{R})$  such that $M_2=QM_1Q^{-1}$ and $D_2=QD_1$,  then $\mu_{M_1,D_1}$ is a spectral measure with spectrum $\Lambda$ if and only if  $\mu_{M_2,D_2}$ is a spectral measure with spectrum $Q^{*-1}\Lambda$.
\end{lemma}

The following result is a known fact, which was proved  in \cite{Dutkay-Haussermann-Lai_2019} and will be used in the  proof of Proposition \ref{th(3.3)}.

\begin{lemma}\label{th(2.3)}
Let $M\in M_n(\mathbb{Z})$ be an expansive integer matrix, and let $D,S\subset\mathbb{Z}^n$ be two finite digit
sets with the same cardinality. Then  the following three affirmations are equivalent:
 \begin{enumerate}[(i)]
 	\item   $(M,D,S)$ is a  Hadamard triple;
 \item $m_{D}(M^{*-1}(s_1-s_2))=0$ for any distinct $s_1, s_2\in S$;
 \item $(\delta_{M^{-1}D},S)$ is a spectral pair.
 \end{enumerate}
\end{lemma}

Recall that $\mu_{M,D}$  is defined by \eqref{1.1}, we  let $A$ be a nonsingular matrix and let the Moran measure
\begin{equation}\label{2.5}
\mu_{A,M,D}=\delta_{A^{-1}D}*\delta_{A^{-1}M^{-1}D}*\delta_{A^{-1}M^{-2}D}*\cdots.
\end{equation}
It is clear that $\mu_{A,M,D}=\mu_{M,D}$ if $A=M$. The following lemma indicates the spectrality of $\mu_{M,D}$ is independent of $A$. The proof is the same as that of \cite[Lemma 2.6]{DHLY} and  \cite[Lemma 3.1]{DL}. For the convenience of readers, we include the proof  here.

\begin{lemma}\label{th(2.4)}
Let $A$ be a nonsingular matrix, and let
$\mu_{A,M,D}$ be  defined by \eqref{2.5}. Then
$$\mu_{M,D}=\mu_{A,M,D}\circ(A^{-1}M).$$
Moreover, $(\mu_{M,D},\Lambda)$ is a spectral pair if and only if $(\mu_{A,M,D},A^*M^{*-1}\Lambda)$ is a spectral pair.
\end{lemma}

\begin{proof}
Applying \eqref{2.1} and \eqref{2.5}, we have
\begin{align}\label{2.6}
\hat{\mu}_{A,M, D}(A^*M^{*-1}\xi)&=m_D(A^{*-1}A^*M^{*-1}\xi)\prod_{j=1}^\infty m_D({M^{*}}^{-{j}}A^{*-1}A^*M^{*-1}\xi) \nonumber \\
&=\prod_{j=1}^\infty m_D(M^{*-j}\xi) =\hat{\mu}_{M, D}(\xi).
\end{align}
Then $\mu_{M,D}=\mu_{A,M,D}\circ(A^{-1}M)$ by the uniqueness of Fourier transform.

Recall that $Q_{\mu,\Lambda}(\xi)$ is defined by \eqref{2.4}, then for $\xi\in\mathbb{R}^2$, it follows from \eqref{2.6} that
\begin{align*}
Q_{\mu_{M,D},\Lambda}(\xi)&=\sum_{\lambda\in\Lambda}|{\hat{\mu}_{M,D}(\xi+\lambda)}|^2 \nonumber \\
&=\sum_{\lambda\in\Lambda}|{\hat{\mu}_{A,M,D}(A^*M^{*-1}(\xi+\lambda))}|^2
\nonumber \\
&=\sum_{\lambda\in\Lambda}|{\hat{\mu}_{A,M,D}(A^*M^{*-1}\xi+A^*M^{*-1}\lambda)}|^2\nonumber \\
&=Q_{\mu_{A,M,D},A^*M^{*-1}\Lambda}(A^*M^{*-1}\xi).
\end{align*}
Hence the second assertion follows by Theorem \ref{th(JP)}.
\end{proof}

We conclude this section by recalling a useful lemma  in our investigation, which was proved by Deng et al. in \cite[Lemma 2.5]{DL}.

\begin{lemma}\label{th(DL)}
Let $p_{i,j}$ be positive numbers such that $\sum\limits_{j=1}^np_{i,j}=1$, and let $q_{i,j}$ be nonnegative numbers such that $\sum\limits_{i=1}^m\max\limits_{1\leq j\leq n} q_{i,j}\leq1$. Then $\sum\limits_{i=1}^m\sum\limits_{j=1}^np_{i,j}q_{i,j}=1$ if and only if $q_{i,1}=\cdots=q_{i,n}$ for $1\leq i\leq m$ and $\sum\limits_{i=1}^mq_{i,1}=1$.
\end{lemma}

\section{\bf Proofs of Theorems \ref{th(1.5)} and \ref{th(1.6)}\label{sect.3}}

In this section, we  focus on proving Theorems \ref{th(1.5)} and \ref{th(1.6)}, that is, studying the spectrality of the measure $\mu_{\tilde{M},\tilde{D}}$, where $\tilde{M}$ and $\tilde{D}$ are given by \eqref{1.4} and \eqref{1.5} respectively.  For this purpose, we  first  give some  properties of $\mathcal{Z}(m_{\tilde{D}})$, and then investigate the structure of the spectrum of $\mu_{\tilde{M},\tilde{D}}$  under the assumption  that $\mu_{A,\tilde{M},\tilde{D}}$  is a spectral measure, where $\mu_{A,\tilde{M},\tilde{D}}$ is defined by \eqref{2.5}.  With these preparations, we will achieve our goal.

According to Lemma \ref{th(2.4)}, without loss of generality,   we  assume in the rest of the paper that
$$A=\begin{pmatrix}
2^{\eta+1}\alpha\beta&0 \\
0&2^{\eta+1}\alpha\beta
\end{pmatrix}.$$
The matrix $A$ will play a special role in the construction of the spectrum of $\mu_{\tilde{M},\tilde{D}}$. Consequently,
\begin{equation}\label{3.1}
\mu_{A,\tilde{M},\tilde{D}}=\delta_{\frac{1}{2^{\eta+1}\alpha\beta}\tilde{D}}*(\mu_{\tilde{M},\tilde{D}}\circ2^{\eta+1}\alpha\beta)\ \ {\rm and} \ \ \hat{\mu}_{A,\tilde{M},\tilde{D}}(\xi)=m_{\tilde{D}}(\frac{\xi}{2^{\eta+1}\alpha\beta})\hat{\mu}_{\tilde{M},\tilde{D}}(\frac{\xi}{2^{\eta+1}\alpha\beta}).
\end{equation}
It is known that
$m_{\tilde{D}}(x)=0$
if and only if
 \begin{equation}\label{3.2}
 \begin{cases}
\alpha x_{1}=\frac{1}{2}+k_1, \\
\omega x_{1}+2^\eta\beta x_{2}=k'_1,
\end{cases}
 \begin{cases}
\alpha x_{1}=k_2, \\
\omega x_{1}+2^\eta\beta x_{2}=\frac{1}{2}+k'_2,
\end{cases}  {\rm or} \ \
\begin{cases}
\alpha x_{1}=\frac{1}{2}+k_3,\\
\omega x_{1}+2^\eta\beta x_{2}=\frac{1}{2}+k'_3,
\end{cases}
\end{equation}
where $k_1, k_2, k_3,k_1^{'}, k_2^{'}, k_3^{'}\in\mathbb{Z}$. By a direct calculation, we have that
\begin{equation}\label{3.3}
\mathcal{Z}(m_{\tilde{D}})=\Theta_1\cup\Theta_2\cup\Theta_3,
\end{equation}
where
\begin{gather*}
\Theta_{1}=\left\{\frac{1}{2^{\eta+1}\alpha\beta}
\begin{pmatrix}
2^\eta(2k_1\beta+\beta) \\
2k_1^{^\prime}\alpha-2k_1\omega-\omega
\end{pmatrix}
  :k_1,k_1^{^\prime}\in\mathbb{Z}
  \right\},\\
\Theta_{2}=\left\{
\frac{1}{2^{\eta+1}\alpha\beta} \begin{pmatrix}
2^{\eta+1}k_2\beta \\
2k_2^{^\prime}\alpha-2k_2\omega+\alpha
\end{pmatrix}
  :k_2,k_2^{^\prime}\in\mathbb{Z}
  \right\},\\
\Theta_{3}=\left\{
\frac{1}{2^{\eta+1}\alpha\beta} \begin{pmatrix}
2^\eta(2k_3\beta+\beta) \\
2k_3^{^\prime}\alpha-2k_3\omega+\alpha-\omega
\end{pmatrix}
  :k_3,k_3^{^\prime}\in\mathbb{Z}
  \right\}.
 \end{gather*}
Define
$$
\Theta_{0}=\left\{\frac{1}{2^{\eta+1}\alpha\beta} \begin{pmatrix}
2^{\eta+1}k_0\beta \\
2k_0^{\prime}\alpha-2k_0\omega
\end{pmatrix}
  :k_0,k_0^{^\prime}\in\mathbb{Z}
  \right\}.
$$
We now make a detailed analysis on the zero set $\mathcal{Z}(m_{\tilde{D}})$ of $m_{\tilde{D}}$.
\begin{pro}\label{th(3.1)}
With the above notations,  the following statements hold.
\begin{enumerate}[(i)]
 	\item  $(\Theta_{i}-\Theta_{i})\cap\mathcal{Z}(m_{\tilde{D}})=\emptyset$ for any $i\in\{0,1,2,3\}$;
 	\item
$\Theta_{i}-\Theta_{j}\subset\mathcal{Z}(m_{\tilde{D}})$ for any distinct $i,j\in\{0,1,2,3\}$;
\item If $\eta=0$, then
$\mathring{\mathcal{F}}_2^2\subset\mathcal{Z}(m_{\tilde{D}})$,
where $\mathring{\mathcal{F}}_2^2$  is defined by \eqref{1.3}.
 \end{enumerate}
\end{pro}
\begin{proof}
(i) Since $\alpha,\beta\in2\mathbb{Z}+1$, from the definitions of $\mathcal{Z}(m_{\tilde{D}})$ and $\Theta_{0}$, it can easily be seen that $\Theta_{i}-\Theta_{i}\subset \Theta_0$ for any $i\in\{0,1,2,3\}$ and $\Theta_{i}\cap\Theta_0=\emptyset$ for any $i\in\{1,2,3\}$. This yields  $(\Theta_{i}-\Theta_{i})\cap\mathcal{Z}(m_{\tilde{D}})=\emptyset$ for all $i$,  which proves (i).

(ii)  For any $\theta_i\in\Theta_{i}$, it is easy to verify that
$$\pm(\theta_i-\theta_0)\in\Theta_{i}~(i\in\{1,2,3\}),\ \ \pm(\theta_1-\theta_2)\in\Theta_{3},\ \ \pm(\theta_1-\theta_3)\in\Theta_{2}\ \ {\rm and}\ \ \pm(\theta_2-\theta_3)\in\Theta_{1}. $$
Hence the assertion follows by using \eqref{3.3}.

(iii) As $\eta=0$ and $\alpha,\beta\in2\mathbb{Z}+1$, it follows from \eqref{3.2} and \eqref{3.3}  that
$$\bigl(\frac{1}{2},0\bigr)^t\in\Theta_{1},\quad \bigl(0,
\frac{1}{2}\bigr)^t\in\Theta_{2}\quad {\rm and} \quad
\bigl(\frac{1}{2},\frac{1}{2}\bigr)^t\in\Theta_{3} $$
if $\omega\in2\mathbb{Z}$, moreover,
$$\bigl(\frac{1}{2},0\bigr)^t\in\Theta_{3},\quad \bigl(0,
\frac{1}{2}\bigr)^t\in\Theta_{2}\quad {\rm and} \quad
\bigl(\frac{1}{2},\frac{1}{2}\bigr)^t\in\Theta_{1} $$
if $\omega\in2\mathbb{Z}+1$. Therefore, $\mathring{\mathcal{F}}_2^2\subset\Theta_1\cup\Theta_2\cup\Theta_3=\mathcal{Z}(m_{\tilde{D}})$.
\end{proof}

\begin{re}\label{re(3.2)}
{\rm Observe that $\alpha,\beta\in2\mathbb{Z}+1$ in $\tilde{D}$, without loss of generality, we can further assume that $\alpha,\beta\geq1$. In fact, if $\alpha<0$ or $\beta<0$, we take
\begin{equation*}
Q=\begin{cases}
\rm{diag}\bigl(-1,1\bigr),  & \mbox{if }\alpha<0,\beta>0; \\
\rm{diag}\bigl(1,-1\bigr),  & \mbox{if }\alpha>0,\beta<0; \\
\rm{diag}\bigl(-1,-1\bigr),  & \mbox{if }\alpha,\beta<0.
\end{cases}
\end{equation*}
Let $\bar{M}=Q\tilde{M}Q^{-1}$
and
$\bar{D}=Q\tilde{D}$. By Lemma \ref{th(2.2)}, we only need to consider the  spectrality of $\mu_{\bar{M},\bar{D}}$. This implies that the assumption is reasonable. }
\end{re}

To investigate the  spectrality of $\mu_{\tilde{M},\tilde{D}}$, we need to  construct a complete residue system of  matrix $A$.
In view of \eqref{3.1} and \eqref{3.3}, one may easily get that
\begin{equation}\label{3.4}
\mathcal{Z}(\hat{\mu}_{A,\tilde{M},\tilde{D}})=\bigcup_{j=0}^{\infty}A^*\tilde{M}^{*j}(\mathcal{Z}(m_{\tilde{D}}))
=\bigcup_{j=0}^{\infty}\tilde{M}^{*j}(2^{\eta+1}\alpha\beta(\Theta_{1}\cup\Theta_{2}\cup\Theta_{3}))\subset\mathbb{Z}^2.
\end{equation}
Throughout this paper, we set $\hbar_p=\{0,1,\ldots,p-1\}$ for an integer $p\geq1$, and let
\begin{equation}\label{3.5}
\mathcal{S}_q=\left\{\begin{pmatrix}
 s_1\\
 s_2\end{pmatrix}:s_1\in \hbar_{2^{q}\beta},s_2\in \hbar_\alpha\right\}\quad {\rm and}\quad \mathcal{T}_q=\bigcup_{i=0}^{3}\mathcal{T}_{q,i},
\end{equation}
where  $q$ is a nonnegative integer and
\begin{gather*}
\mathcal{T}_{q,0}=\left\{\frac{1}{2^{q+1}\alpha\beta} \begin{pmatrix}
2^{q+1}k_0\beta \\
2k_0^{\prime}\alpha-2k_0\omega
\end{pmatrix}
  :k_0\in \hbar_\alpha,k_0^{\prime}\in\hbar_{2^{q}\beta}
  \right\},\\
  \mathcal{T}_{q,1}=\left\{\frac{1}{2^{q+1}\alpha\beta}
\begin{pmatrix}
2^q(2k_1\beta+\beta) \\
2k_1^{\prime}\alpha-2k_1\omega-\omega
\end{pmatrix}
  :k_1\in \hbar_\alpha,k_1^{\prime}\in\hbar_{2^{q}\beta}
  \right\},\\
\mathcal{T}_{q,2}=\left\{\frac{1}{2^{q+1}\alpha\beta} \begin{pmatrix}
2^{q+1}k_2\beta \\
2k_2^{\prime}\alpha-2k_2\omega+\alpha
\end{pmatrix}
  :k_2\in \hbar_\alpha,k_2^{\prime}\in\hbar_{2^{q}\beta}
  \right\},\\
\mathcal{T}_{q,3}=\left\{
\frac{1}{2^{q+1}\alpha\beta} \begin{pmatrix}
2^q(2k_3\beta+\beta) \\
2k_3^{\prime}\alpha-2k_3\omega+\alpha-\omega
\end{pmatrix}
  :k_3\in \hbar_\alpha,k_3^{\prime}\in\hbar_{2^{q}\beta}
  \right\}.
 \end{gather*}

\begin{pro}\label{th(3.3)}
With the above notations,  the following statements hold.
\begin{enumerate}[(i)]
 	\item  $\mathcal{T}_{\eta,i}\subset\Theta_{i}$ for any $i\in\{0,1,2,3\}$;
 \item  $(\delta_{A^{-1}\tilde{D}},\mathcal{C})$ is a spectral pair, where $A=\rm{diag}\bigl(2^{\eta+1}\alpha\beta,2^{\eta+1}\alpha\beta\bigr)$ and
  $\mathcal{C}=2^{\eta+1}\alpha\beta \{\ell_0,\ell_1,\ell_2,\ell_3\}$ for any  $\ell_i\in\mathcal{T}_{\eta,i}$;
 	\item  $\mathcal{S}_\eta\oplus 2^{\eta+1}\alpha\beta\mathcal{T}_\eta$ is a complete residue system of  matrix $A$ in (ii).
 \end{enumerate}
\end{pro}
\begin{proof}
According to the definitions of $\mathcal{T}_{\eta,i}$ and $\Theta_{i}$, (i) is obvious.  We now prove (ii).  In view of Lemma \ref{th(2.3)}, it suffices to prove that $m_{\tilde{D}}(A^{*-1}(c-c^\prime))=0$ for  all distinct $c,c^\prime\in \mathcal{C}$. Since $A=\rm{diag}\bigl(2^{\eta+1}\alpha\beta,2^{\eta+1}\alpha\beta\bigr)$, it follows from  Proposition \ref{th(3.1)}(ii) and Proposition \ref{th(3.3)}(i)  that  $A^{*-1}(c-c^\prime)\in\mathcal{Z}(m_{\tilde{D}})$. This implies  $m_{\tilde{D}}(A^{*-1}(c-c^\prime))=0$, the assertion (ii) follows.

Finally ,we  prove (iii). It is clear that the set $\mathcal{S}_\eta\oplus 2^{\eta+1}\alpha\beta\mathcal{T}_\eta$ can be written as
\begin{align}\label{3.6} \nonumber
\mathcal{S}_\eta\oplus 2^{\eta+1}\alpha\beta\mathcal{T}_\eta
&=
\left\{\begin{pmatrix}
 s_1\\
 s_2\end{pmatrix}:s_1\in \hbar_{2^{\eta}\beta},s_2\in \hbar_\alpha\right\}\oplus
 \begin{pmatrix}
2^{\eta}\beta&0 \\
-\omega&\alpha
\end{pmatrix}\left\{\begin{pmatrix}
 k\\
 k^\prime\end{pmatrix}:
 k\in \hbar_{2\alpha},k^\prime\in \hbar_{2^{\eta+1}\beta}
 \right\}\\
 &:=\mathcal{S}_\eta\oplus\begin{pmatrix}
2^{\eta}\beta&0 \\
-\omega&\alpha
\end{pmatrix}\mathcal{Q}.
\end{align}
To prove that $\mathcal{S}_\eta\oplus 2^{\eta+1}\alpha\beta\mathcal{T}_\eta$ is a complete residue system of  $A={\rm diag}\bigl(2^{\eta+1}\alpha\beta,2^{\eta+1}\alpha\beta\bigr)$, by using \eqref{3.6}, it is sufficient to show that for any $(x,y)^t\in \mathbb{Z}^2$, there exist $(s_1,s_2)^t\in\mathcal{S}_\eta$, $(k,k')^t\in\mathcal{Q}$ and $(x',y^\prime)^t\in \mathbb{Z}^2$ such that
\begin{equation}\label{3.7}
\begin{pmatrix}
 x\\
 y\end{pmatrix}=\begin{pmatrix}
 s_1\\
 s_2\end{pmatrix}+\begin{pmatrix}
2^{\eta}\beta&0 \\
-\omega&\alpha
\end{pmatrix}\begin{pmatrix}
 k\\
 k'\end{pmatrix}+2^{\eta+1}\alpha\beta\begin{pmatrix}
 x'\\
 y'\end{pmatrix}.
\end{equation}
Since $\{0,1,\ldots,2^{\eta}\beta-1\}\oplus2^{\eta}\beta\{0,1,\ldots,2\alpha-1\}$  is a complete residue system of $2^{\eta+1}\alpha\beta$, it follows that there exist $s_1\in\{0,1,\ldots,2^{\eta}\beta-1\}$, $k\in\{0,1,\ldots,2\alpha-1\}$ and $x'\in\mathbb{Z}$ such that
\begin{equation}\label{3.8}
x=s_1+2^{\eta}\beta k+2^{\eta+1}\alpha\beta x'.
\end{equation}
Also note that $\{0,1,\ldots,\alpha-1\}\oplus\alpha\{0,1,\ldots,2^{\eta+1}\beta-1\}$ is another complete residue system of $2^{\eta+1}\alpha\beta$, thus there exist
$s_2\in\{0,1,\ldots,\alpha-1\}$, $k'\in\{0,1,\ldots,2^{\eta+1}\beta-1\}$ and $y'\in\mathbb{Z}$ such that
\begin{equation}\label{3.9}
y+\omega k=s_2+\alpha k'+2^{\eta+1}\alpha\beta y'.
\end{equation}
The above equations \eqref{3.8} and \eqref{3.9} imply that  \eqref{3.7} holds, which concludes the proof.
\end{proof}

Let $\Lambda$ be a spectrum of $\mu_{A,\tilde{M},\tilde{D}}$ with $0\in\Lambda$. By \eqref{2.3} and \eqref{3.4}, we have $\Lambda\subset\mathbb{Z}^2$. This together with Proposition \ref{th(3.3)}(iii) implies that for any $\lambda\in\Lambda$,  there exist some  $s\in\mathcal{S}_\eta$ and $\ell\in\mathcal{T}_\eta$ such that $\lambda=s+2^{\eta+1}\alpha\beta\ell+2^{\eta+1}\alpha\beta\gamma$ for some $\gamma\in\mathbb{Z}^2$. Then for $s\in\mathcal{S}_\eta$ and $\ell\in\mathcal{T}_\eta$, define
\begin{equation}\label{3.10}
\Lambda_{s,\ell}=\left\{\gamma\in\mathbb{Z}^2:s+2^{\eta+1}\alpha\beta\ell+2^{\eta+1}\alpha\beta\gamma\in\Lambda\right\}.
\end{equation}
Then using \eqref{3.5}, we have the following decomposition
\begin{equation}\label{3.11}
\Lambda=\bigcup_{s\in\mathcal{S}_\eta}\bigcup_{i\in\{0,1,2,3\}}
\bigcup_{\ell\in\mathcal{T}_{\eta,i}}(s+2^{\eta+1}\alpha\beta\ell+2^{\eta+1}\alpha\beta\Lambda_{s,\ell}),
\end{equation}
where $s+2^{\eta+1}\alpha\beta\ell+2^{\eta+1}\alpha\beta\Lambda_{s,\ell}=\emptyset$ if $\Lambda_{s,\ell}=\emptyset$.  As $0\in\Lambda$, it follows that
\begin{equation}\label{3.12}
\Lambda_{0,0}\neq\emptyset.
\end{equation}

\begin{lemma}\label{th(3.4)}
Let $\Lambda$ be a spectrum of $\mu_{A,\tilde{M},\tilde{D}}$ with $0\in\Lambda$. If $\Lambda_{s,\ell}$ is a nonempty set,  then $\Lambda_{s,\ell}$ is an orthogonal set of $\mu_{\tilde{M},\tilde{D}}$ for each $s\in\mathcal{S}_\eta$ and $\ell\in\mathcal{T}_\eta$.
\end{lemma}

\begin{proof}
Suppose that $\Lambda_{s,\ell}$ is a nonempty set for $s\in\mathcal{S}_\eta$ and $\ell\in\mathcal{T}_\eta$, then for any distinct $\lambda_1,\lambda_2\in \Lambda_{s,\ell}$, it follows from \eqref{3.11} that $$s+2^{\eta+1}\alpha\beta\ell+2^{\eta+1}\alpha\beta\lambda_1,s+2^{\eta+1}\alpha\beta\ell+2^{\eta+1}\alpha\beta\lambda_2\in \Lambda.$$
Applying \eqref{2.3}, we have $2^{\eta+1}\alpha\beta(\lambda_1-\lambda_2)\in \mathcal{Z}(\hat{\mu}_{A,\tilde{M},\tilde{D}})$. Together with \eqref{3.1}, $\lambda_1,\lambda_2 \in\mathbb{Z}^2$ and $m_{\tilde{D}}(\lambda_1-\lambda_2)=1$, it gives
$$0=\hat{\mu}_{A,\tilde{M},\tilde{D}}(2^{\eta+1}\alpha\beta(\lambda_1-\lambda_2))=
m_{\tilde{D}}(\lambda_1-\lambda_2)
\hat{\mu}_{\tilde{M},\tilde{D}}(\lambda_1-\lambda_2)=\hat{\mu}_{\tilde{M},\tilde{D}}(\lambda_1-\lambda_2).$$
Thus $\lambda_1-\lambda_2\in \mathcal{Z}(\hat{\mu}_{\tilde{M},\tilde{D}})$, which means that $\Lambda_{s,\ell}$ is  an orthogonal set of $\mu_{\tilde{M},\tilde{D}}$.
\end{proof}

The following lemma gives the structure of the spectrum of $\mu_{\tilde{M},\tilde{D}}$  under the assumption  that $\mu_{A,\tilde{M},\tilde{D}}$  is a spectral measure.

\begin{lemma}\label{th(3.5)}
Let $\Lambda$ be a spectrum of $\mu_{A,\tilde{M},\tilde{D}}$ with $0\in\Lambda$. For any
$s\in\mathcal{S}_\eta$, choose a $i_s\in\{0,1,2,3\}$ and
write
\begin{equation*}
\Gamma=\bigcup_{s\in\mathcal{S}_\eta}\bigcup_{\ell\in\mathcal{T}_{\eta,i_s}}
\left(\frac{s+2^{\eta+1}\alpha\beta\ell}{2^{\eta+1}\alpha\beta}+\Lambda_{s,\ell}\right),
\end{equation*}
where $\Lambda_{s,\ell}$ is defined by \eqref{3.10}.
Then $\Gamma$ is a spectrum of $\mu_{\tilde{M},\tilde{D}}$  or an empty set.
\end{lemma}
\begin{proof}
If $\Gamma$ is a nonempty set, we will illustrate our following two steps needed to complete the proof.

{\bf Step 1.} We prove that $\Gamma$ is an orthogonal set of $\mu_{\tilde{M},\tilde{D}}$.

For any distinct $\varsigma_1,\varsigma_2\in \Gamma$, we can write $$\varsigma_k=\frac{s_k+2^{\eta+1}\alpha\beta\ell_k}{2^{\eta+1}\alpha\beta}+\lambda_k,$$ where $s_k\in\mathcal{S}_\eta$,   $\ell_k\in\mathcal{T}_{\eta,i_{s_k}}$, $\lambda_k\in\Lambda_{s_k,\ell_k}$  and $i_{s_k}\in\{0,1,2,3\}$, $k=1,2$. Applying \eqref{3.1}, the fact $\lambda_1,\lambda_2 \in\mathbb{Z}^2$ and the $\mathbb{Z}^2$-periodicity of $m_{\tilde{D}}$, one has
\begin{align}\label{3.13}\nonumber
0&=\hat{\mu}_{A,\tilde{M},\tilde{D}}(2^{\eta+1}\alpha\beta(\varsigma_1-\varsigma_2))
=m_{\tilde{D}}(\varsigma_1-\varsigma_2)\hat{\mu}_{\tilde{M},\tilde{D}}(\varsigma_1-\varsigma_2) \nonumber \\
&=m_{\tilde{D}}(\frac{s_1-s_2}{2^{\eta+1}\alpha\beta}+\ell_1-\ell_2+\lambda_1-\lambda_2)\hat{\mu}_{\tilde{M},\tilde{D}}(\varsigma_1-\varsigma_2) \nonumber\\
&=m_{\tilde{D}}(\frac{s_1-s_2}{2^{\eta+1}\alpha\beta}+\ell_1-\ell_2)\hat{\mu}_{\tilde{M},\tilde{D}}(\varsigma_1-\varsigma_2).
\end{align}
We now claim that $m_{\tilde{D}}(\frac{s_1-s_2}{2^{\eta+1}\alpha\beta}+\ell_1-\ell_2)\neq0$. The proof will be divided into the following two cases.

Case 1: $s_1=s_2$. In this case, it is clear that $\ell_1,\ell_2\in\mathcal{T}_{\eta,i_{s_1}}$ by the definition of $\Gamma$. With Proposition \ref{th(3.1)}(i) and Proposition \ref{th(3.3)}(i), we derive that $\ell_1-\ell_2\notin \mathcal{Z}(m_{\tilde{D}})$. Thus the claim follows.

Case 2:  $s_1\neq s_2$. For this case, we prove the claim by contradiction. Suppose, on the contrary, that
\begin{equation}\label{3.14}
\frac{s_1-s_2}{2^{\eta+1}\alpha\beta}+\ell_1-\ell_2\in \mathcal{Z}(m_{\tilde{D}}).
\end{equation}  By Proposition \ref{th(3.1)} and Proposition \ref{th(3.3)}(i), one has  $\ell_1-\ell_2\in \Theta_{0}\cup\mathcal{Z}(m_{\tilde{D}}).$ Combining this with \eqref{3.14}, we conclude  that
\begin{equation}\label{3.15}
\frac{s_1-s_2}{2^{\eta+1}\alpha\beta}\in \Theta_{0}\cup\mathcal{Z}(m_{\tilde{D}}).
\end{equation}
Using \eqref{3.5} and $s_1\neq s_2$, it is easy to check that $s_1-s_2\in \mathfrak{B}$, where
\begin{equation*}
\mathfrak{B}=\left\{\begin{pmatrix}
t_1 \\
t_2
\end{pmatrix}:t_1\in\left\{1-{2^{\eta}}\beta,\ldots,{2^{\eta}}\beta-1\right\},
t_2\in\left\{1-\alpha,\ldots,\alpha-1\right\}\right\}\setminus
\{\bf0\}.
\end{equation*}
Write $s_1-s_2=(t_1,t_2)^t\in \mathfrak{B}$.
We first prove $t_1=0$. If $t_1\neq0$, it follows  $t_1\notin 2^{\eta}\beta\mathbb{Z}$. Then from the definitions of $\mathcal{Z}(m_{\tilde{D}})$ and $\Theta_{0}$, it can easily be seen that
\begin{equation*}
\frac{s_1-s_2}{2^{\eta+1}\alpha\beta}=\frac{1}{2^{\eta+1}\alpha\beta}\begin{pmatrix}
t_1 \\
t_2
\end{pmatrix}\notin \Theta_{0}\cup\mathcal{Z}(m_{\tilde{D}}).
\end{equation*}
This contradicts to \eqref{3.15}, which proves $t_1=0$.

Since $t_1=0$, it follows from $\beta\in2\mathbb{Z}+1$ that $\frac{s_1-s_2}{2^{\eta+1}\alpha\beta}\notin \Theta_{1}\cup\Theta_{3}$. Together with \eqref{3.15} and $t_1=0$, it yields that
$$\frac{s_1-s_2}{2^{\eta+1}\alpha\beta}=\frac{1}{2^{\eta+1}\alpha\beta}\begin{pmatrix}
0\\
t_2
\end{pmatrix}\in \Theta_{0}\cup\Theta_{2}.$$
By  a simple  calculation, we deduce from $\beta\in2\mathbb{Z}+1$  that $t_2\in \alpha\mathbb{Z}$. However,
$(t_1,t_2)^t=(0,t_2)^t\in \mathfrak{B}$ means that $t_2\notin \alpha\mathbb{Z}$, a contradiction. Hence the claim follows.

Applying  the claim and \eqref{3.13}, we obtain that $\hat{\mu}_{\tilde{M},\tilde{D}}(\varsigma_1-\varsigma_2)=0$. This implies that $\Gamma$ is an orthogonal set of $\mu_{\tilde{M},\tilde{D}}$.

{\bf Step 2.} We prove  the completeness of the exponential function system $E_\Gamma=\{e^{2\pi i\langle\lambda,x\rangle}:\lambda\in\Gamma\}$.

Fix $s\in\mathcal{S}_\eta$, in view of Proposition \ref{th(3.3)}(ii) and Theorem  \ref{th(JP)}, one may get that for any $\ell_{i_s}\in\mathcal{T}_{\eta,i_s}$,
\begin{equation}\label{3.16}
\sum_{{i_s}=0}^3\big|m_{\tilde{D}}(\frac{s+2^{\eta+1}\alpha\beta\ell_{i_s}+\xi}{2^{\eta+1}\alpha\beta})\big|^2\equiv1.
\end{equation}
In \eqref{3.16}, let three of $\ell_0,\ell_1,\ell_2$ and $\ell_3$ be fixed, and the other is altered   in $\mathcal{T}_{\eta,i_s}$, we can easily  verify that for all distinct $\ell,\ell^\prime\in \mathcal{T}_{\eta,i_s}$,
\begin{equation}\label{3.17}
\big|m_{\tilde{D}}(\frac{s+2^{\eta+1}\alpha\beta\ell+\xi}{2^{\eta+1}\alpha\beta})\big|=
\big|m_{\tilde{D}}(\frac{s+2^{\eta+1}\alpha\beta\ell^\prime+\xi}{2^{\eta+1}\alpha\beta})\big|.
\end{equation}
Since $\Lambda_{s,\ell}\subset\mathbb{Z}^2$ and $\Lambda$ is a spectrum of $\mu_{A,\tilde{M},\tilde{D}}$, it follows from the  $\mathbb{Z}^2$-periodicity of $m_{\tilde{D}}(x)$ that
\begin{align}\label{3.18}\nonumber
1&\equiv\sum_{\lambda\in\Lambda}|{\hat{\mu}_{A,\tilde{M},\tilde{D}}(\xi+\lambda)}|^2 \quad   {\rm (by \ Theorem \ \ref{th(JP)})} \nonumber \\
&=\sum_{s\in\mathcal{S}_\eta}\sum_{{i_s}=0}^3\sum_{\ell\in\mathcal{T}_{\eta,i_s}}\sum_{\lambda^\prime\in\Lambda_{s,\ell}}
\big|{\hat{\mu}_{A,\tilde{M},\tilde{D}}(\xi+s+2^{\eta+1}\alpha\beta\ell+2^{\eta+1}\alpha\beta\lambda^\prime)}\big|^2 \quad  {\rm (by \ \eqref{3.11})} \nonumber\\
&=\sum_{s\in\mathcal{S}_\eta}\sum_{{i_s}=0}^3\sum_{\ell\in\mathcal{T}_{\eta,i_s}}\big|m_{\tilde{D}}
(\frac{s+2^{\eta+1}\alpha\beta\ell+\xi}{2^{\eta+1}\alpha\beta})\big|^2
\sum_{\lambda^\prime\in\Lambda_{s,\ell}}\big|{\hat{\mu}_{\tilde{M},\tilde{D}}
(\frac{s+2^{\eta+1}\alpha\beta\ell+\xi}{2^{\eta+1}\alpha\beta}+\lambda^\prime)}\big|^2
\quad {\rm (by \ \eqref{3.1})} \nonumber \\
&=\sum_{s\in\mathcal{S}_\eta}\sum_{{i_s}=0}^3\big|m_{\tilde{D}}(\frac{s+2^{\eta+1}\alpha\beta\ell_{i_s}+\xi}{2^{\eta+1}\alpha\beta})\big|^2
\sum_{\ell\in\mathcal{T}_{\eta,i_s}}\sum_{\lambda^\prime\in\Lambda_{s,\ell}}
\big|{\hat{\mu}_{\tilde{M},\tilde{D}}(\frac{s+2^{\eta+1}\alpha\beta\ell+\xi}
{2^{\eta+1}\alpha\beta}+\lambda^\prime)}\big|^2,
\end{align}
where $\ell_{i_s}\in\mathcal{T}_{\eta,i_s}$ and the last  equality follows from \eqref{3.17}.

We now choose $\xi\in \mathbb{R}^2\setminus\mathbb{Q}^2$, and for simplicity, write
\begin{equation*}
p_{s,{i_s}}=\big|m_{\tilde{D}}(\frac{s+2^{\eta+1}\alpha\beta\ell_{i_s}+\xi}{2^{\eta+1}\alpha\beta})\big|^2 \ \  {\rm and} \ \ q_{s,{i_s}}=\sum_{\ell\in\mathcal{T}_{\eta,i_s}}\sum_{\lambda^\prime\in
\Lambda_{s,\ell}}\big|{\hat{\mu}_{\tilde{M},\tilde{D}}
(\frac{s+2^{\eta+1}\alpha\beta\ell+\xi}{2^{\eta+1}\alpha\beta}+\lambda^\prime)}\big|^2.
\end{equation*}
Then one may derive from \eqref{3.3} that $p_{s,{i_s}}>0$, and \eqref{3.18} becomes
\begin{equation}\label{3.19}
\sum_{s\in\mathcal{S}_\eta}\sum_{{i_s}=0}^3p_{s,{i_s}}q_{s,{i_s}}=1.
\end{equation}
Note that $\Gamma$ is an orthogonal set of $\mu_{\tilde{M},\tilde{D}}$, thus Theorem \ref{th(JP)} implies that
$$
\sum_{s\in\mathcal{S}_\eta}\max\{q_{s,0},q_{s,1},q_{s,2},q_{s,3}\}\leq1.
$$
Together with \eqref{3.16}, \eqref{3.19}  and Lemma \ref{th(DL)}, it concludes that
\begin{equation}\label{3.20}
\sum_{s\in\mathcal{S}_\eta}\sum_{\ell\in\mathcal{T}_{\eta,i_s}}
\sum_{\lambda^\prime\in\Lambda_{s,\ell}}\big|{\hat{\mu}_{\tilde{M},\tilde{D}}
(\frac{s+2^{\eta+1}\alpha\beta\ell+\xi}{2^{\eta+1}\alpha\beta}+\lambda^\prime)}\big|^2=1,\quad {i_s}=0,1,2,3,
\end{equation}
and
\begin{align}\label{3.21}\nonumber
\sum_{\ell\in\mathcal{T}_{\eta,0}}\sum_{\lambda^\prime\in\Lambda_{s,\ell}}\big|{\hat{\mu}_{\tilde{M},\tilde{D}}
(\frac{s+2^{\eta+1}\alpha\beta\ell+\xi}
{2^{\eta+1}\alpha\beta}+\lambda^\prime)}\big|^2&=\sum_{\ell\in\mathcal{T}_{\eta,1}}
\sum_{\lambda^\prime\in\Lambda_{s,\ell}}\big|{\hat{\mu}_{\tilde{M},\tilde{D}}(\frac{s+2^{\eta+1}\alpha\beta\ell+\xi}
{2^{\eta+1}\alpha\beta}+\lambda^\prime)}\big|^2 \nonumber \\
&=\sum_{\ell\in\mathcal{T}_{\eta,2}}\sum_{\lambda^\prime\in\Lambda_{s,\ell}}\big|{\hat{\mu}_{\tilde{M},\tilde{D}}
(\frac{s+2^{\eta+1}\alpha\beta\ell+\xi}
{2^{\eta+1}\alpha\beta}+\lambda^\prime)}\big|^2 \nonumber\\
&=\sum_{\ell\in\mathcal{T}_{\eta,3}}\sum_{\lambda^\prime\in\Lambda_{s,\ell}}\big|{\hat{\mu}_{\tilde{M},\tilde{D}}
(\frac{s+2^{\eta+1}\alpha\beta\ell+\xi}
{2^{\eta+1}\alpha\beta}+\lambda^\prime)}\big|^2
\end{align}
for any $s\in\mathcal{S}_\eta$.

By continuity, we conclude that the above equations \eqref{3.20} and \eqref{3.21} hold for all $\xi\in \mathbb{R}^2$. Therefore,  Theorem \ \ref{th(JP)} shows that $\Gamma$ a spectrum of $\mu_{\tilde{M},\tilde{D}}$ for any group $\{i_s\}_{s\in\mathcal{S}_\eta}$ with $i_s\in\{0,1,2,3\}$.
The  proof  is complete.
\end{proof}

\begin{re}\label{re(3.6)}
{\rm Suppose $\Lambda=\bigcup_{s\in\mathcal{S}_\eta}\bigcup_{i\in\{0,1,2,3\}}
\bigcup_{\ell\in\mathcal{T}_{\eta,i}}(s+2^{\eta+1}\alpha\beta\ell+2^{\eta+1}
\alpha\beta\Lambda_{s,\ell})$ is a spectrum of $\mu_{A,\tilde{M},\tilde{D}}$ with $0\in\Lambda$, then we can conclude from \eqref{3.21} that for any $s\in\mathcal{S}_\eta$, one of the following two statements holds:
\begin{enumerate}[(i)]
 	\item  There exist some $\ell_{i_s}\in\mathcal{T}_{\eta,i_s}$ such that $\Lambda_{s,\ell_{i_s}}\neq\emptyset$ for all $0\leq i_s\leq 3$;
 	\item  $\Lambda_{s,\ell}=\emptyset$ for any $\ell\in\mathcal{T}_\eta=\bigcup_{i=0}^{3}\mathcal{T}_{\eta,i}$.
 \end{enumerate}
 In particular,
the assumption $0\in\Lambda$ implies $\Lambda_{0,0}\neq\emptyset$. Therefore, (i)  always hold for $s=0$, which illustrates that
there must exist
$\ell_{i_0}\in\mathcal{T}_{\eta,i_0}$ such that $\Lambda_{0,\ell_{i_0}}\neq\emptyset$ for all $1\leq i_0\leq 3$.}
\end{re}

In order to prove Theorems \ref{th(1.5)} and \ref{th(1.6)} more conveniently, we define
\begin{gather*}
\Phi_{0}=\left\{\upsilon\in \mathbb{Z}^2:\upsilon=
(0,0)^t
\pmod {2\mathbb{Z}^2}
\right\},\\
\Phi_{1}=\left\{\upsilon\in \mathbb{Z}^2:\upsilon=
(1,0)^t
\pmod {2\mathbb{Z}^2}
\right\},\\
\Phi_{2}=\left\{\upsilon\in \mathbb{Z}^2:\upsilon=
(0,1)^t
\pmod {2\mathbb{Z}^2}
\right\},\\
\Phi_{3}=\left\{\upsilon\in \mathbb{Z}^2:\upsilon=
(1,1)^t
\pmod {2\mathbb{Z}^2}
\right\}.
\end{gather*}
Then
 \begin{equation}\label{3.22}
\mathbb{Z}^2=\bigcup_{i=0}^3\Phi_i.
\end{equation}
We have all ingredients for the proof of Theorem \ref{th(1.5)}.

\begin{proof}[ Proof of  Theorem \ref{th(1.5)}]
We will prove this theorem by the circle ``$(ii)\Rightarrow(iii)\Rightarrow(iv)\Rightarrow(i)\Rightarrow(ii)$".

``$(ii)\Rightarrow(iii)$"
If $\tilde{M}\in M_2(2\mathbb{Z})$, we can write $\tilde{M}=\begin{pmatrix}
2\tilde{a}&2\tilde{b} \\
2\tilde{c}&2\tilde{d}
\end{pmatrix}$ with $\tilde{ a},\tilde{b},\tilde{c},\tilde{d}\in\mathbb{Z}$. Then with \eqref{1.3}, it is easy to verify that
$$\tilde{M}\mathring{\mathcal{F}}_2^2=
\left\{\begin{pmatrix}
\tilde{a}\\
\tilde{c}\end{pmatrix},\begin{pmatrix}
\tilde{b}\\
\tilde{d}\end{pmatrix},\begin{pmatrix}
\tilde{a}+\tilde{b}\\
\tilde{c}+\tilde{d}\end{pmatrix}\right\}\subset\mathbb{Z}^2.$$
Hence the assertion follows.

``$(iii)\Rightarrow(iv)$"
Suppose $\tilde{M}\mathring{\mathcal{F}}_2^2\subset\mathbb{Z}^2$, which implies $\tilde{C}:=\tilde{M}^*\mathcal{F}_2^2\subset\mathbb{Z}^2$. Then using Lemma \ref{th(2.3)} and Proposition \ref{th(3.1)}(iii),
we obtain that  $(\tilde{M},\tilde{D},\tilde{C})$ is a  Hadamard triple. Therefore, $(\tilde{M},\tilde{D})$ is admissible.

``$(iv)\Rightarrow(i)$" If $(\tilde{M},\tilde{D})$ is admissible, then $\mu_{\tilde{M},\tilde{D}}$ is a spectral measure  by Theorem \ref{th(DHL)}.

``$(i)\Rightarrow(ii)$" Suppose that $\mu_{\tilde{M},\tilde{D}}$ is a spectral measure, and let $A={\rm diag}\bigl(2\alpha\beta,2\alpha\beta\bigr)$. In view of Lemma \ref{th(2.4)}, one may derive that $\mu_{A,\tilde{M},\tilde{D}}$ is also a spectral measure. Let $\Lambda$ be a spectrum of $\mu_{A,\tilde{M},\tilde{D}}$ with $0\in\Lambda$.
First, we  construct a spectrum of $\mu_{\tilde{M},\tilde{D}}$.
Recall that $\mathcal{T}_{\eta,i}$ and $\Phi_i$  are defined by \eqref{3.5} and \eqref{3.22}, respectively. By $\eta=0$ and a simple calculation, one has $2\alpha\beta \tilde{M}^{*}\mathcal{T}_{\eta,0}\subset \Phi_{0}$. For $i\in\{1,2,3\}$, we can suppose that $2\alpha\beta \tilde{M}^{*}\mathcal{T}_{\eta,i}\subset \Phi_{j_i}$ for some $j_i\in\{0,1,2,3\}$. Consequently,
\begin{equation*}
\bigcup_{i=1}^3 2\alpha\beta \tilde{M}^{*}\mathcal{T}_{\eta,i}\subset \bigcup_{i=1}^3\Phi_{j_i}.
\end{equation*}
This means that for any $s\in\mathcal{S}_\eta\setminus\{0\}$,
there exists  ${i_s}\in\{0,1,2,3\}$ such that $s+2\alpha\beta\ell_s\notin\bigcup_{j=1}^3 2\alpha\beta \tilde{M}^{*}\mathcal{T}_{\eta,j}+2\mathbb{Z}^2$ for any $\ell_s\in\mathcal{T}_{\eta,i_s}$.
Define
\begin{equation}\label{3.23}
\Gamma=\Delta_{0,0}\cup\bigcup_{s\in\mathcal{S}_\eta\setminus\{{0}\}}
\Delta_{s,i_s},
\end{equation}
where $\Delta_{0,0}=\bigcup_{\ell_0\in\mathcal{T}_{\eta,0}}(\ell_0+\Lambda_{0,\ell_0})$,
 $\Delta_{s,i_s}=
\bigcup_{\ell_s\in\mathcal{T}_{\eta,i_s}}
(\frac{s+2\alpha\beta\ell_s}{2\alpha\beta}+\Lambda_{s,\ell_s})$ with
\begin{equation}\label{3.24}
(s+2\alpha\beta\mathcal{T}_{\eta,i_s})\cap\left(\bigcup_{j=1}^3 2\alpha\beta \tilde{M}^{*}\mathcal{T}_{\eta,j}+2\mathbb{Z}^2\right)=\emptyset,
\end{equation}
 and $\Lambda_{s,\ell_s}$ is defined by \eqref{3.10}.
In view of Lemma \ref{th(3.5)}, we get that $\Gamma$
is a spectrum of $\mu_{\tilde{M},\tilde{D}}$. Moreover, it follows from $0\in\Lambda$ and Lemma \ref{th(2.4)}  that $0\in\Gamma$.

Second, we prove that for any $i\in\{1,2,3\}$, there must exist $\ell_i\in\mathcal{T}_{\eta,i}$ such that $2\alpha\beta \tilde{M}^*\ell_i\in2\mathbb{Z}^2$. Since  $\Gamma$
is a spectrum of $\mu_{\tilde{M},\tilde{D}}$ with $0\in \Gamma$, it follows from Lemma \ref{th(2.4)}  that $2\alpha\beta \tilde{M}^{*-1}\Gamma$ is a spectrum of $\mu_{A,\tilde{M},\tilde{D}}$ with $0\in2\alpha\beta \tilde{M}^{*-1}\Gamma$. Using \eqref{3.11}, one has
\begin{equation}\label{3.25}
2\alpha\beta \tilde{M}^{*-1}\Gamma=\bigcup_{s^\prime\in\mathcal{S}_\eta}\bigcup_{i\in\{0,1,2,3\}}
\bigcup_{\ell_{i}^\prime\in\mathcal{T}_{\eta,i}}(s^\prime+2\alpha\beta\ell_{i}^\prime+
2\alpha\beta\Lambda^\prime_{s^\prime,\ell_{i}^\prime}),
\end{equation}
where
$$
\Lambda^\prime_{s^\prime,\ell_{i}^\prime}=\left\{\gamma\in\mathbb{Z}^2:s^\prime+2\alpha\beta\ell_{i}^\prime+
2\alpha\beta\gamma\in2\alpha\beta \tilde{M}^{*-1}\Gamma\right\}.
$$
For $s^\prime=0$ and $\ell^\prime_{i}=0\in\mathcal{T}_{\eta,0}$, we have $\Lambda^\prime_{0,0}\neq\emptyset$ since $0\in2\alpha\beta \tilde{M}^{*-1}\Gamma$.  By Remark \ref{re(3.6)}, there must exist
$\ell^\prime_{i}\in\mathcal{T}_{\eta,i}$ such that $\Lambda^\prime_{0,\ell_{i}^\prime}\neq\emptyset$ for all $1\leq i\leq 3$. Let $\lambda_{i}^\prime\in\Lambda^\prime_{0,\ell_{i}^\prime}$, where  $i=1,2, 3$. Therefore, \eqref{3.23} and \eqref{3.25} imply that
there exist $s_{i}\in\mathcal{S}_\eta$,
$\ell_{i}\in\bigcup_{j=0}^3\mathcal{T}_{\eta,j}$ and
$\lambda_{i}\in\Lambda_{s_i,\ell_i}$ such that $\frac{s_i+2\alpha\beta\ell_i}{2\alpha\beta}+\lambda_{i}\in \Gamma$ and
\begin{equation}\label{3.26}
2\alpha\beta \tilde{M}^*\ell_{i}^\prime+2\alpha\beta \tilde{M}^*\lambda_{i}^\prime
=s_{i}+2\alpha\beta\ell_{i}+2\alpha\beta\lambda_{i}\quad {\rm for} \ {i}=1,2,3.
\end{equation}
Moreover, it follows from \eqref{3.24} that
$s_i+2\alpha\beta\ell_i\notin\bigcup_{j=1}^3 2\alpha\beta \tilde{M}^{*}\mathcal{T}_{\eta,j}+2\mathbb{Z}^2$ if $s_i\neq 0$ for $i=1,2,3.$  However, by noting that  $\lambda_i,\lambda_i^\prime\in\mathbb{Z}^2$, then \eqref{3.26} yields that
$$
s_i+2\alpha\beta\ell_i\in2\alpha\beta \tilde{M}^*\ell_i^\prime+ 2\mathbb{Z}^2\subset 2\alpha\beta \tilde{M}^{*}\mathcal{T}_{\eta,i}+2\mathbb{Z}^2
\subset\bigcup_{j=1}^3 2\alpha\beta \tilde{M}^{*}\mathcal{T}_{\eta,j}+2\mathbb{Z}^2
$$
for $i=1,2,3.$ Therefore, the above discussion shows that $s_i=0$ for $i=1,2,3$, and hence $\ell_i\in\mathcal{T}_{\eta,0}$ by the definition of $\Gamma$.
This implies that $2\alpha\beta\ell_{i}\in2\mathbb{Z}^2$ for ${i}=1,2,3$.
Combining this with $\tilde{M}\in M_2(\mathbb{Z})$, $s_i=0$ and $\lambda_{i},\lambda_{i}^\prime\in\mathbb{Z}^2$, one may infer from \eqref{3.26} that
\begin{equation*}
2\alpha\beta \tilde{M}^*\ell_{i}^\prime
=2\alpha\beta\ell_{i}+2\alpha\beta(\lambda_{i}
-\tilde{M}^*\lambda_{i}^\prime)\in2\mathbb{Z}^2\quad {\rm for} \ {i}=1,2,3.
\end{equation*}
Therefore, $2\alpha\beta \tilde{M}^*\ell_i^\prime\in2\mathbb{Z}^2$ for some $\ell_i^\prime\in\mathcal{T}_{\eta,i}$, where $i=1,2,3$.

It remains to prove $\tilde{M}\in M_2(2\mathbb{Z})$. For any $i\in\{1,2,3\}$, the above conclusion shows that there must exist $\ell_i\in\mathcal{T}_{\eta,i}$ such that $2\alpha\beta \tilde{M}^*\ell_i\in2\mathbb{Z}^2$. For these $\ell_i\in\mathcal{T}_{\eta,i}$, $i=1,2,3$,  according to the definition of $\mathcal{T}_{\eta,i}$ and the fact $\alpha,\beta\in2\mathbb{Z}+1$, it can easily be checked that
$$\left\{2\alpha\beta\ell_i:i=1,2,3\right\}=\left\{\begin{pmatrix}
1\\
0\end{pmatrix},\begin{pmatrix}
0\\
1\end{pmatrix},\begin{pmatrix}
1\\
1\end{pmatrix}\right\}\pmod {2\mathbb{Z}^2}.$$
This together with $2\alpha\beta \tilde{M}^*\ell_i\in2\mathbb{Z}^2$ and a simple calculation gives that $\tilde{M}^*\in M_2(2\mathbb{Z})$, which is equivalent to $\tilde{M}\in M_2(2\mathbb{Z})$.  This finishes the proof of Theorem \ref{th(1.5)}.
\end{proof}

The following lemma plays an important role in the proof  of  Theorem \ref{th(1.6)}.

\begin{lemma}\label{th(3.7)}
Let $\mu_{\tilde{M},\tilde{D}}$ be a spectral measure, where $\tilde{M}$ and $\tilde{D}$ are given by \eqref{1.4} and \eqref{1.5} respectively. If $\eta>0$ in $\tilde{D}$, then $\tilde{M}=\begin{pmatrix}
\tilde{ a}&\tilde{b} \\
\tilde{c}&\tilde{d}
\end{pmatrix}$ satisfies  $2^{\eta+1}|\tilde{c}$.
\end{lemma}
\begin{proof}
Suppose, on the contrary, that $2^{\eta+1}\nmid \tilde{c}$. Then one may write $\tilde{c}=2^\tau c^\prime$ for some integer $\tau\leq \eta$ and $c^\prime\in 2\mathbb{Z}+1$. Let $Q_1={\rm diag}\bigl(1,\frac{1}{2^\tau}\bigl)$. A simple calculation gives
\begin{equation*}
M_1:=Q_1\tilde{M}Q_1^{-1}
=\begin{pmatrix}
\tilde{a}&2^\tau \tilde{b} \\
c^\prime&\tilde{d}
\end{pmatrix}\in M_2(\mathbb{Z})
\end{equation*}
and
\begin{equation*}
D_1:=Q_1\tilde{D}=\left\{ \begin{pmatrix}
0\\0\end{pmatrix},\begin{pmatrix}
\alpha\\0\end{pmatrix},\begin{pmatrix}
\omega \\ 2^{\eta-\tau}\beta\end{pmatrix},\begin{pmatrix}
-\alpha-\omega \\
-2^{\eta-\tau}\beta\end{pmatrix}\right\}\subset\mathbb{Z}^2,
\end{equation*}
where $\alpha,\beta\in2\mathbb{Z}+1$. Since $\mu_{\tilde{M},\tilde{D}}$ is a spectral measure, it follows from Lemmas \ref{th(2.2)} and \ref{th(2.4)} that  $\mu_{M_1,D_1}$ and $\mu_{A_1,M_1,D_1}$ are also  spectral measures, where $A_1={\rm diag}\bigl(2^{\eta-\tau+1}\alpha\beta,2^{\eta-\tau+1}\alpha\beta\bigl)$ and $\mu_{A_1,M_1,D_1}$ is defined by \eqref{2.5}.

If $\tau=\eta$, it follows from  Theorem \ref{th(1.5)} that $M_1\in M_2(2\mathbb{Z})$. This means that $c^\prime\in 2\mathbb{Z}$, a contradiction. Hence the assertion follows.

If $\tau<\eta$, we derive the contradiction by constructing a spectrum of $\mu_{{M}_1,D_1}$. Recall that
$\mathcal{S}_{\eta-\tau}$ and $\mathcal{T}_{\eta-\tau}=\bigcup_{i=0}^{3}\mathcal{T}_{\eta-\tau,i}$
are defined by \eqref{3.5}, we first prove the following two claims.

\noindent
\emph{{\bf Claim 1.} Let $\Phi_1$ and $\Phi_3$  be given by \eqref{3.22}. Then
\begin{equation*}
2^{\eta-\tau+1}\alpha\beta {M}_1^{*}\mathcal{T}_{\eta-\tau,2}\subset\begin{cases}
\Phi_{1},  & \mbox{if }\tilde{d}\in2\Bbb Z; \\
\Phi_{3}, & \mbox{if }\tilde{d}\in2\Bbb Z+1.
\end{cases}
\end{equation*}}
\begin{proof}[Proof of Claim 1]
For any $\ell\in\mathcal{T}_{\eta-\tau,2}$, there exist
$k\in \hbar_\alpha$ and $k^{\prime}\in\hbar_{2^{\eta-\tau}\beta}$ such that
\begin{equation}\label{3.27}
\ell=\frac{1}{2^{\eta-\tau+1}\alpha\beta} \begin{pmatrix}
2^{\eta-\tau+1}k\beta \\
2k^{\prime}\alpha-2k\omega+\alpha
\end{pmatrix}.
\end{equation}
Since $M_1
=\begin{pmatrix}
\tilde{a}&2^\tau \tilde{b} \\
c^\prime&\tilde{d}
\end{pmatrix}$, $\tau<\eta$ and $\alpha,c^\prime\in 2\mathbb{Z}+1$, it follows from  \eqref{3.27} that
\begin{equation*}
2^{\eta-\tau+1}\alpha\beta {M}_1^{*}\ell=\begin{pmatrix}
2(2^{\eta-\tau}k\tilde{a}\beta+(k^{\prime}\alpha-k\omega)c^\prime)+c^\prime\alpha \\
2(2^{\eta}k\tilde{b}\beta+(k^{\prime}\alpha-k\omega)\tilde{d})+\tilde{d}\alpha
\end{pmatrix} =\begin{pmatrix}
1 \\
\tilde{d}
\end{pmatrix}\pmod {2\mathbb{Z}^2}.
\end{equation*}
Consequently, $2^{\eta-\tau+1}\alpha\beta {M}_1^{*}\ell\in \Phi_{1}$ if $\tilde{d}\in2\Bbb Z$, and $2^{\eta-\tau+1}\alpha\beta {M}_1^{*}\ell\in \Phi_{3}$ if $\tilde{d}\in2\Bbb Z+1$.
So the claim follows.
\end{proof}

\noindent
\emph{{\bf Claim 2.} Let $\Phi_1$ and $\Phi_3$  be given by \eqref{3.22}. Then for any $s\in\mathcal{S}_{\eta-\tau}\setminus\{0\}$, the following two statements hold.
\begin{enumerate}[(i)]
 	\item  There exist some ${i_s}\in\{0,1,2,3\}$ such that $s+2^{\eta-\tau+1}\alpha\beta\ell_s\notin\Phi_{1}$  for any $\ell_s\in\mathcal{T}_{\eta-\tau,i_s}$;
 	\item  There exist some ${i_s}\in\{0,1,2,3\}$ such that $s+2^{\eta-\tau+1}\alpha\beta\ell_s\notin\Phi_{3}$  for any $\ell_s\in\mathcal{T}_{\eta-\tau,i_s}$.
 \end{enumerate}
}
\begin{proof}[Proof of Claim 2]
Begin by observing that $\alpha\in 2\mathbb{Z}+1$ and $\tau<\eta$, then for any $\ell_i\in\mathcal{T}_{\eta-\tau,i}$, $i=0,1,2,3$, we have
$$2^{\eta-\tau+1}\alpha\beta\ell_0=\begin{pmatrix}
2^{\eta-\tau+1}k_0\beta \\
2k_0^{\prime}\alpha-2k_0\omega
\end{pmatrix}=\begin{pmatrix}
0 \\
0
\end{pmatrix}\pmod {2\mathbb{Z}^2},$$
$$2^{\eta-\tau+1}\alpha\beta\ell_1=\begin{pmatrix}
2^{\eta-\tau}(2k_1\beta+\beta) \\
2k_1^{\prime}\alpha-2k_1\omega-\omega
\end{pmatrix}=\begin{pmatrix}
0 \\
\omega
\end{pmatrix}\pmod {2\mathbb{Z}^2},$$
$$2^{\eta-\tau+1}\alpha\beta \ell_2=\begin{pmatrix}
2^{\eta-\tau+1}k_2\beta \\
2k_2^{\prime}\alpha-2k_2\omega+\alpha
\end{pmatrix}=\begin{pmatrix}
0 \\
1
\end{pmatrix}\pmod {2\mathbb{Z}^2}
$$
and
$$
2^{\eta-\tau+1}\alpha\beta \ell_3=\begin{pmatrix}
2^{\eta-\tau}(2k_3\beta+\beta) \\
2k_3^{\prime}\alpha-2k_3\omega+\alpha-\omega
\end{pmatrix}=\begin{pmatrix}
0 \\
\omega-1
\end{pmatrix}\pmod {2\mathbb{Z}^2}$$
for some
$k_i\in \hbar_\alpha$ and $k_i^{\prime}\in\hbar_{2^{\eta-\tau}\beta}$.
Without loss of generality, we assume that $\omega\in 2\mathbb{Z}$ (the case $\omega\in 2\mathbb{Z}+1$ can be similarly proved). Then a simple calculation gives
 \begin{equation}\label{3.28}
 2^{\eta-\tau+1}\alpha\beta\ell_0, 2^{\eta-\tau+1}\alpha\beta\ell_1\in\Phi_{0}\quad {\rm and}\quad
2^{\eta-\tau+1}\alpha\beta\ell_2,2^{\eta-\tau+1}\alpha\beta\ell_3\in\Phi_{2}.
\end{equation}
Recall that $\mathcal{T}_{\eta-\tau}=\bigcup_{i=0}^{3}\mathcal{T}_{\eta-\tau,i}$. Then  for any $s=(s_1,s_2)^t\in\mathcal{S}_{\eta-\tau}\setminus\{0\}$,   we take
\begin{equation*}
\ell_s\in\begin{cases}
\mathcal{T}_{\eta-\tau},  & \mbox{if }s_1\in2\Bbb Z; \\
\mathcal{T}_{\eta-\tau,2}\cup\mathcal{T}_{\eta-\tau,3},  & \mbox{if }s_1\in2\Bbb Z+1,s_2\in2\Bbb Z;\\
\mathcal{T}_{\eta-\tau,0}\cup\mathcal{T}_{\eta-\tau,1},  & \mbox{if }s_1,s_2\in2\Bbb Z+1.
\end{cases}
\end{equation*}
This together with \eqref{3.28} yields that $s+2^{\eta-\tau+1}\alpha\beta\ell_s\notin\Phi_{1}$, which proves (i).   For (ii),
we take
\begin{equation*}
\ell_s\in\begin{cases}
\mathcal{T}_{\eta-\tau},  & \mbox{if }s_1\in2\Bbb Z; \\
\mathcal{T}_{\eta-\tau,0}\cup\mathcal{T}_{\eta-\tau,1},  & \mbox{if }s_1\in2\Bbb Z+1,s_2\in2\Bbb Z;\\
\mathcal{T}_{\eta-\tau,2}\cup\mathcal{T}_{\eta-\tau,3},  & \mbox{if }s_1,s_2\in2\Bbb Z+1.
\end{cases}
\end{equation*}
Consequently, $s+2^{\eta-\tau+1}\alpha\beta\ell_s\notin\Phi_{3}$ by \eqref{3.28}.
Thus Claim 2 follows.
\end{proof}

We now continue with the proof of the case $\tau<\eta$. In the following proof, we might as well assume $\tilde{d}\in2\Bbb Z$ in $M_1$. If $\tilde{d}\in2\Bbb Z+1$, we only need to replace Claim 2(i) with Claim 2(ii).

Since $\tau<\eta$ and $\tilde{d}\in2\Bbb Z$, it follows from Claim 2(i) that for any $s\in\mathcal{S}_{\eta-\tau}\setminus\{0\}$, there must exist some ${i_s}\in\{0,1,2,3\}$ such that $s+2^{\eta-\tau+1}\alpha\beta\ell_s\notin\Phi_{1}$ for any $\ell_s\in\mathcal{T}_{\eta-\tau,i_s}$.
Let $\tilde{\Lambda}$ be a spectrum of $\mu_{A_1,M_1,D_1}$ with $0\in\tilde{\Lambda}$. Define
\begin{equation*}
\tilde{\Gamma}=\tilde{\Delta}_{0,0}\cup\bigcup_{s\in\mathcal{S}_{\eta-\tau}\setminus\{{0}\}}
\tilde{\Delta}_{s,i_s},
\end{equation*}
where $\tilde{\Delta}_{0,0}=\bigcup_{\ell_0\in\mathcal{T}_{\eta-\tau,0}}(\ell_0+\tilde{\Lambda}_{0,\ell_0})$, $\tilde{\Delta}_{s,i_s}=
\bigcup_{\ell_s\in\mathcal{T}_{\eta-\tau,i_s}}
(\frac{s+2^{\eta-\tau+1}\alpha\beta\ell_s}{2^{\eta-\tau+1}\alpha\beta}
+\tilde{\Lambda}_{s,\ell_s})$ with $$(s+2^{\eta-\tau+1}\alpha\beta\mathcal{T}_{\eta-\tau,i_s})
\cap\Phi_{1}=\emptyset,$$ and
$$\tilde{\Lambda}_{s,\ell_s}=\left\{\gamma\in\mathbb{Z}^2:s+2^{\eta-\tau+1}\alpha\beta\ell_s+2^{\eta-\tau+1}\alpha\beta\gamma\in\tilde{\Lambda}\right\}.$$
Using the similar argument as in the proof of Lemma \ref{th(3.5)}, we can  carry out that $\tilde{\Gamma}$
is a spectrum of $\mu_{{M}_1,D_1}$ with $0\in\tilde{\Gamma}$.

Next, we prove that  there must exist  $\ell\in\mathcal{T}_{\eta-\tau,2}$ such that $2^{\eta-\tau+1}\alpha\beta {M}_1^*\ell\in2\mathbb{Z}^2$. Since $\tilde{\Gamma}$
is a spectrum of $\mu_{{M}_1,D_1}$ with $0\in\tilde{\Gamma}$, it follows from Lemma \ref{th(2.4)}  that $2^{\eta-\tau+1}\alpha\beta M_1^{*-1}\tilde{\Gamma}$ is a spectrum of $\mu_{A_1,{M}_1,D_1}$ with $0\in2^{\eta-\tau+1}\alpha\beta {M}_1^{*-1}\tilde{\Gamma}$. Similar to \eqref{3.25}, we have that
\begin{equation*}
2^{\eta-\tau+1}\alpha\beta {M}_1^{*-1}\tilde{\Gamma}=\bigcup_{s^\prime\in\mathcal{S}_{\eta-\tau}}\bigcup_{i\in\{0,1,2,3\}}
\bigcup_{\ell_{i}^\prime\in\mathcal{T}_{\eta-\tau,i}}(s^\prime+2^{\eta-\tau+1}\alpha\beta\ell_{i}^\prime+
2^{\eta-\tau+1}\alpha\beta\tilde{\Lambda}^\prime_{s^\prime,\ell_{i}^\prime}),
\end{equation*}
where
$$
\tilde{\Lambda}^\prime_{s^\prime,\ell_{i}^\prime}=\left\{\gamma\in\mathbb{Z}^2:s^\prime+2^{\eta-\tau+1}\alpha\beta\ell_{i}^\prime+
2^{\eta-\tau+1}\alpha\beta\gamma\in2^{\eta-\tau+1}\alpha\beta {M}_1^{*-1}\tilde{\Gamma}\right\}.
$$
For $s^\prime=0$ and $\ell^\prime_{i}=0\in\mathcal{T}_{\eta-\tau,0}$, it follows from $0\in2^{\eta-\tau+1}\alpha\beta {M}_1^{*-1}\tilde{\Gamma}$ that $\tilde{\Lambda}^\prime_{0,0}\neq\emptyset$. Similar to Remark \ref{re(3.6)}, one may infer that there  exists
$\ell^\prime_{2}\in\mathcal{T}_{\eta-\tau,2}$ such that $\tilde{\Lambda}^\prime_{0,\ell_{2}^\prime}\neq\emptyset$. Therefore, applying Claim 1 and the similar argument as in the proof of Theorem \ref{th(1.5)}, we can easily conclude that $2^{\eta-\tau+1}\alpha\beta {M}_1^*\ell^\prime_{2}\in2\mathbb{Z}^2$. Thus the assertion follows.

Finally, we prove $2^{\eta+1}| \tilde{c}$. The above discussion means that there exist some $\ell\in\mathcal{T}_{\eta-\tau,2}$ such that $2^{\eta-\tau+1}\alpha\beta {M}_1^*\ell\in2\mathbb{Z}^2$. For these $\ell\in\mathcal{T}_{\eta-\tau,2}$, it follows from \eqref{3.27} that
 \begin{equation*}
2^{\eta-\tau+1}\alpha\beta {M}_1^{*}\ell=\begin{pmatrix}
2(2^{\eta-\tau}k\tilde{a}\beta+(k^{\prime}\alpha-k\omega)c^\prime)+c^\prime\alpha \\
2(2^{\eta}k\tilde{b}\beta+(k^{\prime}\alpha-k\omega)\tilde{d})+\tilde{d}\alpha
\end{pmatrix}
\end{equation*}
for some $k\in \hbar_\alpha$ and $k^{\prime}\in\hbar_{2^{\eta-\tau}\beta}$.
Together with $2^{\eta-\tau+1}\alpha\beta {M}_1^*\ell\in2\mathbb{Z}^2$, it yields that $c^\prime\alpha\in2\mathbb{Z}$. This contradicts the fact  $c^\prime,\alpha\in2\mathbb{Z}+1$, and hence the assumption $2^{\eta+1}\nmid \tilde{c}$ does not hold. Therefore, we obtain $2^{\eta+1}| \tilde{c}$, and complete the proof.
\end{proof}

Having established the above preparation, now we are in a position to prove Theorem \ref{th(1.6)}.

\begin{proof}[ Proof of  Theorem \ref{th(1.6)}]
We first prove the necessity. Suppose  $\mu_{\tilde{M},\tilde{D}}$ is a spectral measure. In view of Lemma \ref{th(3.7)}, we have that $\tilde{M}=\begin{pmatrix}
\tilde{ a}&\tilde{b} \\
\tilde{c}&\tilde{d}
\end{pmatrix}$ satisfies  $2^{\eta+1}|\tilde{c}$. Thus one may write $\tilde{c}=2^{\eta+1}\kappa$ with $\kappa\in \mathbb{Z}$. Let $\tilde{Q}={\rm diag}\bigl(1,\frac{1}{2^\eta}\bigl)$. By a simple calculation, we get
\begin{equation}\label{3.29}
\bar{M}:=\tilde{Q}\tilde{M}\tilde{Q}^{-1}
=\begin{pmatrix}
\tilde{a}&2^\eta \tilde{b} \\
2\kappa&\tilde{d}
\end{pmatrix}
\end{equation}
and
\begin{equation}\label{3.30}
\bar{D}:=\tilde{Q}\tilde{D}=\left\{ \begin{pmatrix}
0\\0\end{pmatrix},\begin{pmatrix}
\alpha\\0\end{pmatrix},\begin{pmatrix}
\omega \\ \beta\end{pmatrix},\begin{pmatrix}
-\alpha-\omega \\
-\beta\end{pmatrix}\right\}.
\end{equation}
Since $\mu_{\tilde{M},\tilde{D}}$ is a spectral measure,  it follows from Lemma \ref{th(2.2)} that  $\mu_{\bar{M},\bar{D}}$ is also a spectral measure. Then with Theorem \ref{th(1.5)}, we have $\bar{M}\in M_2(2\mathbb{Z})$.
This together with \eqref{3.29} gives that $\tilde{a},\tilde{d}\in2\Bbb Z$. Hence the necessity follows.

Now we are devoted to proving the sufficiency.
Suppose $\tilde{M}=\begin{pmatrix}
\tilde{ a}&\tilde{b} \\
\tilde{c}&\tilde{d}
\end{pmatrix}$, where $\tilde{a},\tilde{d}\in2\Bbb Z$ and $2^{\eta+1}| \tilde{c}$. Then there exist  $a^*,c^*,d^*\in\mathbb{Z}$ such that $ \tilde{a}=2a^*$, $\tilde{c}=2^{\eta+1}c^*$ and $ \tilde{d}=2d^* $.  Let $\tilde{Q}={\rm diag}\bigl(1,\frac{1}{2^\eta}\bigr)$. A simple calculation gives
\begin{equation*}
M^\prime:=\tilde{Q}\tilde{M}\tilde{Q}^{-1}
=\begin{pmatrix}
2a^*&2^\eta \tilde{b} \\
 2c^*&2d^*
\end{pmatrix},
\end{equation*}
and  $\bar{D}=\tilde{Q}\tilde{D}$  is given by \eqref{3.30}.
Since $\eta>0$, it follows from Theorem \ref{th(1.5)} that $\mu_{M^\prime,\bar{D}}$ is  a spectral measure. Therefore, $\mu_{\tilde{M},\tilde{D}}$ is  a spectral measure by Lemma \ref{th(2.2)}.

This completes the proof of Theorem \ref{th(1.6)}.
\end{proof}

\section{\bf Proofs of Theorems \ref{th(main)} and \ref{th(main2)}}

In the present section, we  are committed to investigating the spectrality of the measure $\mu_{M,D}$, where $M\in M_2(\Bbb Z)$ is an expansive integer  matrix and $D$ is given by \eqref{1.2}. We first prove  Theorem \ref{th(main)} by using Theorems \ref{th(1.5)} and \ref{th(1.6)}, and then prove Theorem \ref{th(main2)}. Finally, we  provide some concluding remarks.

\begin{proof}[Proof of Theorem \ref{th(main)}]
The sufficiency follows directly from Theorem \ref{th(DHL)} and Lemma \ref{th(2.2)}. Now we are devoted to proving the necessity. Suppose that $\mu_{M,D}$ is a spectral measure. Let $\eta=\max\left\{r:2^r|(\alpha_1\beta_2-\alpha_2\beta_1)\right\}$,
and let $\tilde{M}$ and $\tilde{D}$ be given by \eqref{1.4} and \eqref{1.5} respectively. That is, $\tilde{M}=QMQ^{-1}$
and $\tilde{D}=QD$.
In view of Lemma \ref{th(2.2)}, $\mu_{\tilde{{M}},\tilde{{D}}}$ is  a spectral measure. It suffices to prove that there exists a matrix  $\tilde{Q}\in M_2(\mathbb{R})$  such that  $(\bar{M},\bar{D})$ is admissible, where $\bar{M}=\tilde{Q}\tilde{M}\tilde{Q}^{-1}$
and $\bar{D}=\tilde{Q}\tilde{D}$. The proof will be divided into the following two cases.

Case 1: $\eta=0$.  Since $\mu_{\tilde{{M}},\tilde{{D}}}$  is a spectral measure, it follows from $\eta=0$ and Theorem \ref{th(1.5)} that $(\tilde{{M}},\tilde{{D}})$ is admissible. Thus the assertion follows by taking $\tilde{Q}={\rm diag}\bigl(1,1\bigr)$.

Case 2:  $\eta>0$. Note that $\mu_{\tilde{{M}},\tilde{{D}}}$  is a spectral measure, thus Theorem \ref{th(1.6)} implies that one may write $\tilde{M}=\begin{pmatrix}
2a^\prime&b^\prime \\
2^{\eta+1}c^\prime&2d^\prime
\end{pmatrix}$, where $a^\prime,b^\prime,c^\prime,d^\prime,\in\Bbb Z$.
We take $\tilde{Q}={\rm diag}\bigl(1,\frac{1}{2^\eta}\bigr)$. Then $$\bar{M}=\tilde{Q}\tilde{M}\tilde{Q}^{-1}=\begin{pmatrix}
2a^\prime&2^{\eta}b^\prime \\
2c^\prime&2d^\prime
\end{pmatrix} \quad {\rm and} \quad \bar{D}=\tilde{Q}\tilde{D}=\left\{ \begin{pmatrix}
0\\0\end{pmatrix},\begin{pmatrix}
\alpha\\0\end{pmatrix},\begin{pmatrix}
\omega \\ \beta\end{pmatrix},\begin{pmatrix}
-\alpha-\omega \\
-\beta\end{pmatrix}\right\}.$$
Using $\eta>0$, it is clear that $\bar{M}\in M_2(2\Bbb Z)$. Hence $(\bar{M},\bar{D})$ is admissible by Theorem \ref{th(1.5)}.

This completes the proof of Theorem \ref{th(main)}.
\end{proof}

Next, we focus on proving Theorem \ref{th(main2)}.
\begin{proof}[Proof of Theorem \ref{th(main2)}]
Let $\tilde{M}$ and $\tilde{D}$ be given by \eqref{1.4} and \eqref{1.5}, respectively. That is, \begin{equation}\label{4.1}
\tilde{M}=QMQ^{-1}\quad {\rm and} \quad\tilde{D}=QD,
\end{equation}
where the matrix $Q\in M_2(\Bbb Z)$ satisfies $\det(Q)=1$. In view of Lemma \ref{th(2.2)}, $\mu_{M,D}$ is a spectral measure if and only if $\mu_{\tilde{M},\tilde{D}}$ is a spectral measure. This implies that Theorem \ref{th(main2)}(i) is equivalent to Theorem \ref{th(1.5)}(i). Note that $\det(Q)=1$, by a simple calculation, one has that
$$
M\in M_2(2\Bbb Z)\Longleftrightarrow\tilde{M}\in M_2(2\Bbb Z).
$$
Thus Theorem \ref{th(main2)}(ii) and (iii) are  equivalent to Theorem  \ref{th(1.5)}(ii) and (iii), respectively. Finally, from the Definition \ref{deA} and \eqref{4.1}, it is easy to see that
$(\tilde{M},\tilde{D})$ is admissible $\Longleftrightarrow$ there exists a set $ \tilde{C}\subset\mathbb{Z}^2$ such that
$(\tilde{M},\tilde{D},\tilde{C})$ is a Hadamard triple $\Longleftrightarrow$
$(M,D,Q^*\tilde{C})$ is a Hadamard triple $\Longleftrightarrow$ $(M,D)$ is admissible. Consequently, Theorem \ref{th(main2)}(iv) is equivalent to Theorem  \ref{th(1.5)}(iv).

Therefore, the desired result now is obtained  by appeal to Theorem  \ref{th(1.5)}.
\end{proof}

At the end of this paper, we give some further remarks and list an open question which related to our main results.  The following example is specifically used to display our results, which are convenient to judge whether the measure $\mu_{M,D}$ in question {\bf (Qu 1)} is a spectral measure.

\begin{ex} \label{th(4.2)}
{\rm Let  $M_1=\begin{pmatrix}
2&b \\
2&2
\end{pmatrix}$ and
$M_2=\begin{pmatrix}
2&b\\
4&2
  \end{pmatrix}$ be two expansive integer matrices, and let
\begin{equation*}
D_1=\left\{\begin{pmatrix}
0\\0\end{pmatrix},\begin{pmatrix}
1\\ 0
\end{pmatrix},
\begin{pmatrix}
0\\ 1
\end{pmatrix},\begin{pmatrix}
-1\\ -1
\end{pmatrix}\right\} \quad {\rm and} \quad
D_2=\left\{\begin{pmatrix}
0\\0\end{pmatrix},\begin{pmatrix}
1\\ 0
\end{pmatrix},
\begin{pmatrix}
0\\ 2
\end{pmatrix},\begin{pmatrix}
-1\\ -2
\end{pmatrix}\right\}.
\end{equation*}
Then the following statements hold.
 \begin{enumerate}[(i)]
 \item $\mu_{M_1,D_1}$ and $\mu_{M_2,D_1}$ are spectral measures if and only if $b\in 2\mathbb{Z}$;
 	\item   $\mu_{M_1,D_2}$ is a non-spectral measure, while $\mu_{M_2,D_2}$ is a spectral  measure.
 \end{enumerate}
}
\end{ex}
\begin{proof}
By a simple calculation, the result follows directly from Theorems \ref{th(1.5)} and \ref{th(1.6)}.
\end{proof}

It is worth noting that if $\alpha_1\beta_2-\alpha_2\beta_1\in 2\Bbb Z$ in Theorem \ref{th(main)}, we cannot give the specific form of matrix $M$. However, if $\alpha_1,\alpha_2,\beta_1$ and $\beta_2$ are fixed, we can describe the specific form by applying Theorem  \ref{th(1.6)}. The following simple but interesting example is devoted to illustrating this fact.

\begin{ex} \label{th(4.3)}
{\rm Let $M=\begin{pmatrix}
a&b \\
c&d
\end{pmatrix}$ be an expansive integer matrix, and let
\begin{equation*}
D=\left\{\begin{pmatrix}
0\\0\end{pmatrix},\begin{pmatrix}
1\\ 2
\end{pmatrix},
\begin{pmatrix}
3\\ 8
\end{pmatrix},\begin{pmatrix}
-4\\ -10
\end{pmatrix}\right\}.
\end{equation*}
Then $\mu_{M,D}$ is a spectral  measure if and only if $a,d\in 2\mathbb{Z}$ and $c\in 4\mathbb{Z}$.
}
\end{ex}
\begin{proof}
Write $Q=\begin{pmatrix}
 3&-1 \\
 -2&1
\end{pmatrix}$. Then it is direct to compute that
\begin{equation*}
\tilde{M}:=QMQ^{-1}
=\begin{pmatrix}
3a-c+2(3b-d) & 3a-c+3(3b-d) \\
c-2a+2(d-2b) & c-2a+3(d-2b)
\end{pmatrix}
\end{equation*}
and
\begin{equation*}
\tilde{D}:=QD=\left\{ \begin{pmatrix}
0\\0\end{pmatrix},\begin{pmatrix}
1\\0\end{pmatrix},\begin{pmatrix}
1 \\ 2\end{pmatrix},\begin{pmatrix}
-2 \\
-2\end{pmatrix}\right\}.
\end{equation*}
By Lemma \ref{th(2.2)},  $\mu_{M,D}$ is a spectral  measure if and only if $\mu_{\tilde{{M}},\tilde{{D}}}$ is  a spectral measure.

For the sufficiency, it follows from $a,d\in 2\mathbb{Z}$ and $c\in 4\mathbb{Z}$ that there exist $\tilde{a},\tilde{c},\tilde{d}\in\mathbb{Z}$ such that $a=2\tilde{a},d=2\tilde{d}$ and $c= 4\tilde{c}$. Thus $\tilde{M}$ becomes $$\tilde{M}
=\begin{pmatrix}
2(3\tilde{a}-2\tilde{c}+3b-d) & 3a-c+3(3b-d) \\
4(\tilde{c}-\tilde{a}+\tilde{d}-b) & 2(2\tilde{c}-a+3\tilde{d}-b)
\end{pmatrix}.$$
This together with Theorem \ref{th(1.6)} yields that $\mu_{\tilde{M},\tilde{D}}$ is a spectral measure, and hence the sufficiency follows.

Conversely, suppose $\mu_{\tilde{M},\tilde{D}}$ is a spectral measure. Applying Theorem \ref{th(1.6)}, we have
\begin{equation*}
\begin{cases}
3a-c+2(3b-d)\in2\Bbb Z, \\
c-2a+2(d-2b)\in4\Bbb Z, \\ c-2a+3(d-2b)\in2\Bbb Z.
\end{cases}
\end{equation*}
Consequently,
$3a-c,c+3d\in2\Bbb Z$ and $c-2a+2d\in4\Bbb Z$.
By a simple calculation, we infer that $a,d\in 2\mathbb{Z}$ and $c\in 4\mathbb{Z}$. This proves the necessity.
\end{proof}

We remark here that the digit set $D$ in \eqref{1.2} satisfies $\alpha_1\beta_2-\alpha_2\beta_1\neq 0$, it is of interest to consider the following question:

{\bf (Qu 2):} For an expansive integer matrix $M\in M_2(\mathbb{Z})$  and the digit set $D$ given by \eqref{1.2} with $\alpha_1\beta_2-\alpha_2\beta_1=0$, what is the sufficient and necessary condition for $\mu_{M,D}$ to be a spectral measure?

To the best of our knowledge, so far there is no sign of solving this problem. This seems to be a more difficult question. An answer to the question {\bf(Qu 2)}  may shed some light on the study of the spectrality of  fractal measures.

\end{document}